\documentclass[10.5pt]{article}
\usepackage[leqno]{amsmath}
\usepackage{amsfonts}
\usepackage{graphicx}

\usepackage{amsmath}
\usepackage{amssymb}
\usepackage{latexsym}
\usepackage{amsmath, amsfonts,amssymb, amsthm, euscript,makeidx,color,mathrsfs}

\def\sqr#1#2{{\vcenter{\vbox{\hrule height.#2pt
              \hbox{\vrule width.#2pt height#1pt \kern#1pt \vrule width.#2pt}
              \hrule height.#2pt}}}}

\def\3n{\negthinspace \negthinspace \negthinspace }
\def\2n{\negthinspace \negthinspace }
\def\1n{\negthinspace }

\def\dbA{\mathbb{A}}
\def\dbB{\mathbb{B}}
\def\dbC{\mathbb{C}}

\def\dbE{\mathbb{E}}
\def\dbF{\mathbb{F}}
\def\dbG{\mathbb{G}}
\def\dbH{\mathbb{H}}

\def\dbP{\mathbb{P}}

\def\dbR{\mathbb{R}}
\def\dbS{\mathbb{S}}

\def\dbY{\mathbb{Y}}
\def\dbZ{\mathbb{Z}}

\def\sA{\mathscr{A}}
\def\sB{\mathscr{B}}
\def\sC{\mathscr{C}}
\def\sD{\mathscr{D}}

\def\sG{\mathscr{G}}
\def\sH{\mathscr{H}}

\def\sM{\mathscr{M}}
\def\sN{\mathscr{N}}

\def\sP{\mathscr{P}}
\def\sQ{\mathscr{Q}}
\def\sR{\mathscr{R}}

\def\sX{\mathscr{X}}
\def\sY{\mathscr{Y}}
\def\sZ{\mathscr{Z}}

\def\={\buildrel \triangle \over =}

\def\ds{\displaystyle}

\def\ns{\noalign{\ss}}
\def\g{\gamma}
\def\d{\delta}
\def\e{\varepsilon}

\def\si{\sigma}

\def\f{\varphi}
\def\th{\theta}

\def\Th{\Theta}
\def\L{\Lambda}

\def\cA{{\cal A}}

\def\cC{{\cal C}}

\def\cF{{\cal F}}
\def\cG{{\cal G}}
\def\cH{{\cal H}}

\def\cK{{\cal K}}
\def\cL{{\cal L}}
\def\cM{{\cal M}}
\def\cN{{\cal N}}

\def\cP{{\cal P}}

\def\cR{{\cal R}}

\def\cX{{\cal X}}
\def\cY{{\cal Y}}
\def\cZ{{\cal Z}}

\def\ss{\smallskip}
\def\ms{\medskip}

\def\q{\quad}
\def\qq{\qquad}
\def\hb{\hbox}

\def\lan{\mathop{\langle}}
\def\ran{\mathop{\rangle}}

\def\h{\widehat}
\def\wt{\widetilde}

\def\cd{\cdot}

\def\({\Big (}
\def\){\Big )}
\def\[{\Big[}
\def\]{\Big]}
\def\bde{\begin{definition}}
\def\ede{\end{definition}}
\def\be{\begin{equation}}
\def\bel{\begin{equation}\label}
\def\ee{\end{equation}}
\def\bt{\begin{theorem}}
\def\et{\end{theorem}}
\def\bc{\begin{corollary}}
\def\ec{\end{corollary}}
\def\bl{\begin{lemma}}
\def\el{\end{lemma}}
\def\bp{\begin{proposition}}
\def\ep{\end{proposition}}
\def\bas{\begin{assumption}}
\def\eas{\end{assumption}}
\def\br{\begin{remark}}
\def\er{\end{remark}}
\def\ba{\begin{array}}
\def\ea{\end{array}}
\def\ed{\end{document}}

\def\square#1{\vbox{\hrule\hbox{\vrule height#1%
     \kern#1\vrule}\hrule}}
\def\rectangle#1#2{\vbox{\hrule\hbox{\vrule height#1%
     \kern#2\vrule}\hrule}}

\font\tenbb=msbm10 \font\sevenbb=msbm7 \font\fivebb=msbm5

\newfam\bbfam
\scriptscriptfont\bbfam=\fivebb \textfont\bbfam=\tenbb
\scriptfont\bbfam=\sevenbb

\newtheorem{lemma}{Lemma}[section]
\newtheorem{remark}{Remark}[section]

\newtheorem{theorem}{Theorem}[section]
\newtheorem{corollary}{Corollary}[section]

\newtheorem{definition}{Definition}[section]
\newtheorem{proposition}{Proposition}[section]
\newtheorem{assumption}{Assumption}[section]

\makeatletter
   
   \@addtoreset{equation}{section}
\makeatother

\begin{document}
 \title{\bf Characterizations of equilibrium controls in time inconsistent mean-field stochastic linear quadratic
problems. I \footnote{The research was supported by the NSF of China under grant 11231007, 11401404 and 11471231.}}

\author{Tianxiao Wang \footnote{School of
Mathematics, Sichuan University, Chengdu, P. R. China. Email:wtxiao2014@scu.edu.cn.}}

\maketitle

\begin{abstract}
In this paper, a class of time inconsistent linear quadratic optimal control problems of mean-field stochastic differential equations (SDEs) is considered under Markovian framework. Open-loop equilibrium controls and their particular closed-loop representations are introduced and characterized via variational ideas. Several interesting features are revealed and a system of coupled Riccati equations is derived. In contrast with the analogue optimal control problems of SDEs, the mean-field terms in state equation, which is another reason of time inconsistency, prompts us to define above two notions in new manners. An interesting result, which is almost trivial in the counterpart problems of SDEs, is given and plays significant role in the previous characterizations. As application, the uniqueness of open-loop equilibrium controls is discussed.
\end{abstract}


\ms


\section{Introduction}

Suppose $(\Omega,\cF,\dbP,\dbF)$ is a complete filtered probability space,
$W(\cd)$ is a one-dimensional standard Brownian motion with natural
filtration $\dbF\equiv\{\cF_t\}_{t\ge0}$ augmented by all $\dbP$-null sets.
For any $t\in[0,T)$, we consider the
following stochastic differential equation (SDE):
\bel{state-equation-original}\left\{\2n\ba{ll}
\ns\ds
dX(s)=\big[A(s)X(s)+B(s)u(s)  \big]ds +\big[C(s) X(s) +D(s)  u(s) \big]dW(s),\q
s\in[t,T],\\
\ns\ds X(t)=\xi.\ea\right.\ee
Here $A,B,C,D$ are suitable matrix-valued
(deterministic) functions, $X(\cd)$, $u(\cd)$, $(t,\xi)\in\sD$ is called the {\it state process}, {\it control process}, {\it initial pair}, respectively, where
$\sD:=\Big\{(t,\xi)\bigm|t\in[0,T],~\xi\hb{ is $\cF_t$-measurable, }\dbE|\xi|^2<\infty\Big\}.$
Under some mild
conditions, for any $(t,\xi)$
and control $u(\cd)$, (\ref{state-equation-original})
admits a unique solution $X(\cd)=X(\cd\,;t,\xi,u(\cd))$. The classical stochastic linear quadratic optimal control problems is to find suitable $\bar u(\cd)=\bar u(
\cd;t,\xi)$ to minimize the following cost functional
\bel{cost-original-classical-static}\ba{ll}
\ns\ds J(u(\cd);t,\xi)={1\over2}\dbE_t\Big\{\int_t^T\big[\lan
Q(s)X(s),X(s)\ran +\lan
R(s)u(s),u(s)\ran\big]ds +\lan GX(T),X(T)\ran\Big\},\ea\ee
where $Q,  R,G$ are suitable matrix-valued
(deterministic) functions, $\dbE_t(\cd):=\dbE[\,\cd\,|\cF_t]$ stands for conditional expectation operator. For the optimal control, we observe that one fundamental property is the time consistency, i.e., for optimal control $\bar u(\cd)\equiv \bar u(\cd;t,\bar X(t))$, one has $\bar u(s;t_1,\bar X(t_1))=\bar u(s;t_2,\bar X(t_2))$ with $t\leq t_1\leq t_2\leq s\leq T$, $X(t_2)=X(t_2;t_1,\bar X(t_1),\bar u(\cd))$.

Inspired by the formulation of mean-variance portfolio selection problems, it is reasonable to keep the state process of above optimal control problem stable with respect to possible variation of random factors. One effective way is to add the variation of $X(\cd)$, i.e.
$$\dbE_t\big[X(s)-\dbE_t X(s)\big]^2=\dbE_t |X(s)|^2-\big[\dbE_t X(s)\big]^2,\ \ s\in(t,T].
$$
into the cost functional, and we end up with the form of
\bel{ }\ba{ll}
\ns\ds J( u(\cd);t,\xi)={\,1\over2}\,\dbE_t\Big\{\int_t^T\big[ \lan
\wt Q(s) \dbE_tX(s) ,\dbE_tX(s) \ran+\lan
Q(s)X(s),X(s)\ran +
\lan R(s) u(s) ,u(s)\ran\big]ds \\
\ns\ds\qq\qq\qq\qq  +\lan GX(T),X(T)\ran+\lan\wt
G\dbE_t[X(T)],\dbE_t[X(T)]\ran\Big\}.\ea\ee
Under proper conditions, optimal control of the form $\bar u=\Th_1 \bar X+\Th_2\dbE_t\bar X$ exists with appropriate $\Th_i$, see e.g. Section 3 of \cite{Yong-2017}. Plugging it into (\ref{state-equation-original}), we arrive at one conditional mean-field SDEs for optimal state $\bar X$,
\bel{state-equation-2}\left\{\2n\ba{ll}
\ns\ds
d\bar X(s)=\big[A_1(s)\bar X(s)  +A_2(s) \dbE_t \bar X(s)  \big]ds +\big[A_3(s) \bar X(s) +A_4(s)\dbE_t \bar X(s)\big]dW(s),\q
s\in[t,T],\\
\ns\ds X(t)=\xi,\ea\right.\ee
with proper $A_i$. To give a unified treatment on both (\ref{state-equation-original}) and (\ref{state-equation-2}), for $t\in[0,T)$, we propose the following controlled mean-field SDE,
\bel{standard-state-equation}\left\{\2n\ba{ll}
\ns\ds
dX\!=\!\big[A X \! +\!\wt A \dbE_t X \!+\!B  u\!+\!\wt B \dbE_t u \!+\!b \big]ds \!+\!\big[C X \!+\!\wt C \dbE_t X
\!+D u\!  +\!\wt D \dbE_t u\!+\!\si \big]dW(s),\q
s\in[t,T],\\
\ns\ds
dX\!=\!\big[(A+\wt A) X \!  +(\!B +\wt B) u\!+\!b \big]ds\!+\!\big[(C+\wt C)X \!+(D+\wt D) u\!+\!\si \big]dW(s),\q
s\in[0,t],\ \ t\in[0,T),\\
\ns\ds X(0)=x.\ea\right.\ee
Here and after, the time reference may be omitted for simplicity. The solvability of (\ref{standard-state-equation}) is easy to see if moreover, $\wt A, \wt B, \wt C, \wt D$ are bounded and deterministic, $b,\si$ are proper processes. We also consider the following quadratic cost functional
\bel{cost-functional-1}\ba{ll}
\ns\ds J(u(\cd);t,X(t))={\,1\over2}\,\dbE_t\Big\{\int_t^T\big[ \lan
\wt Q \dbE_tX ,\dbE_tX \ran+
\lan \wt R \dbE_tu,\dbE_tu\ran+\lan
QX,X\ran +
\lan R u ,u\ran\big]ds\\
\ns\ds\qq\qq\qq\q +\lan GX(T),X(T)\ran +\lan\wt
G\dbE_t[X(T)],\dbE_t[X(T)]\ran\Big\}+\lan \gamma_1 X(t)+\gamma_2,\dbE_t X(T)\ran,
\ea\ee
which is obviously well-defined. Our linear quadratic optimal control problem can be stated as follows.

\ms

\bf Problem (LQ). \rm For any given $(t,X(t))\in\sD$, to find $\bar u(\cd)\in L^2_{\dbF}(0,T;\dbR^m)$ such that
\bel{1.3}J(\bar
u\big|_{[t,T]}(\cd);t,X(t))=\inf_{u(\cd)\in L^2_{\dbF}(t,T;\dbR^m)}J(u(\cd);t,X(t)).\ee

If $t=0$, Problem (LQ) was studied in \cite{Yong-2013-SICON}, (see also \cite{Buckdahn-Djehiche-Li-2011}, \cite{Li-Automatic}, \cite{Li-Sun-Yong-2017}) and the optimal control exists under proper conditions. Returning back to above $dynamic$ setting, any optimal control $\bar u(\cd) $ associated with $(t,X(t))$ satisfying (\ref{1.3})
will depend on $t$ and demonstrate the time-inconsistency
property, i.e. $\bar u(s;t_1,\bar X(t_1))\neq\bar u(s;t_2,\bar X(t_2))$ for some $(t_1,t_2,s)$ with $t\leq t_1\leq t_2\leq s\leq T.$ In other words, to solve Problem (LQ), one has to make the choice between ``optimality'' and ``time consistency''.
In most existing papers along this line, the time consistency was kept, and the traditional closed-loop optimal controls, open-loop optimal controls were replaced by closed-loop equilibrium controls, open-loop equilibrium controls, respectively.
As to closed-loop equilibrium controls, we refer the reader to e.g., \cite{Bjork-Murgoci-2017}, \cite{Wang-Wu-2016}, \cite{Yong-2011-MCRF}, \cite{Yong-2017}, where some delicate convergence arguments from discrete time to continuous case were used, and \cite{Huang-Li-Wang-2017}, \cite{Wang-2017-JDE}, where a new approach based on variational ideas were developed without convergence procedures.
We also refer to \cite{Ekeland-Mbodji-Pirvu-2012}, \cite{Ekeland-Pirvu-2008} for the corresponding study of investment and consumption problems with non-exponential discounting.
On the other hand, there were also many articles on open-loop equilibrium controls, see e.g., \cite{Hu-Jin-Zhou-2012}, \cite{Hu-Jin-Zhou-2017}, \cite{Wang-Wu-2015}, \cite{Wei-Wang-2017}, \cite{Yong-2017}, and so on.
We point out that almost all the previous literature on time inconsistent stochastic linear quadratic problems focused on the particular case of $\wt A=\wt B=\wt C=\wt D=0$, except \cite{Yong-2017} where the closed-loop equilibrium controls of Problem (LQ) were introduced and studied via multi-person differential games approach.
To our best, the investigation on open-loop equilibrium controls of Problem (LQ) is still open. To fill this gap, in this paper we introduce two notions, i.e., open-loop equilibrium controls and their closed-loop representations, of Problem (LQ), and establish their characterizations by the variational ideas in \cite{Huang-Li-Wang-2017}, \cite{Wang-2017-JDE}. As application, we discuss the uniqueness of open-loop equilibrium controls.

There are several essential differences between the existing papers and ours.
In contrast with \cite{Hu-Jin-Zhou-2012}, \cite{Hu-Jin-Zhou-2017}, \cite{Wang-Wu-2015}, \cite{Wang-2017-JDE}, \cite{Yong-2017}, our state equation is a general conditional mean-field SDE. The additional mean-field terms is the second reason of time inconsistency, and requires us to propose new definitions of equilibrium controls and new mathematical tricks, see e.g. Lemma \ref{Lemma-equality}.
Even under the particular SDEs case, our obtained second-order equilibrium conditions did not appear in \cite{Hu-Jin-Zhou-2012}, \cite{Hu-Jin-Zhou-2017}, \cite{Yong-2017}, \cite{Alia-Chighoub-Sohail-2016}. For the proof of uniqueness of open-loop equilibrium controls, our result extends the counterparts in \cite{Hu-Jin-Zhou-2017}, and our procedures are different from theirs as well.
We emphasize that the characterization viewpoint on time inconsistent stochastic linear quadratic problem were also used in other specific/different frameworks, such as \cite{Djehiche-Huang-2016}, \cite{Hu-Jin-Zhou-2017}, \cite{Huang-Li-Wang-2017}, \cite{Wang-Wu-2015}, \cite{Wang-2017-JDE}.
At last, by our study we also find the following interesting facts:

$\bullet$ The open-loop equilibrium controls are characterized by two kinds of conditions: $first$-$order$, $second$-$order$ $equilibrium$ $conditions$, which is comparable with the $first$-$order$, $second$-$order$ $necessary$ $optimality$ $conditions$ in traditional optimal control problems.

$\bullet$ The second-order equilibrium condition is the same as the second-order optimality condition of mean-field SDEs, and it appears in both open-loop equilibrium controls and their closed-loop representations.

$\bullet$ As to the closed-loop representations of open-loop equilibrium controls, the first-order equilibrium condition includes a system of Riccati equations, which appears for the first time and are essentially different from that of closed-loop equilibrium controls in \cite{Yong-2017}.

The article is organized as follows. In Section 2, we introduce some useful spaces, as well the notions of open-loop equilibrium controls, and their closed-loop representations. In Section 3, we characterize both notions by variational approach. In Section 4, we discuss the uniqueness of open-loop equilibrium controls under proper conditions. Section 5 concludes this paper.

\section{Preliminaries}

We first introduce the following hypotheses.

\ms

\bf (H1) \rm For the coefficients in (\ref{standard-state-equation}) and (\ref{cost-functional-1}), suppose
$$\ba{ll}
\ns\ds  A(\cd),\wt A(\cd), C(\cd),\wt C(\cd)\in  L^{\infty}(0,T;\dbR^{n\times n}),\ \ B(\cd),\wt B(\cd),D(\cd),\wt D(\cd)\in  L^{\infty}(0,T;\dbR^{n\times
m}),\\
\ns\ds  Q(\cd),\wt Q(\cd) \in  L^{\infty}(0,T;\dbS^{n\times n}),\ \ G,\wt G\in\dbS^{n\times n},\ \ \gamma_1\in\dbR,\ \gamma_2\in\dbR^n,\ \   \\
\ns\ds b(\cd)\in L^2(\Omega;L^1(0,T;\dbR^{n})), \ \  \si(\cd)\in L^2_{\dbF}(0,T;\dbR^n),\ \ R(\cd),\wt R(\cd)\in  L^{\infty}(0,T;\dbS^{m\times
m}).
\ea
$$
Here $\dbS^{m\times m}$ is the set of symmetric $m\times m$ matrices. For $0\leq s\leq t\leq T$, $H:=\dbR^n,\dbR^{n\times n},$ etc, we define the following spaces. $L^2_{\cF_t}(\Omega;H)$ is the set of $\cF_t$-measurable random variables $X:\Omega\rightarrow H$ such that $\dbE|X|^2<\infty$;
$L^{\infty}(s,t; H)$ is the set of deterministic, measurable, essentially bounded functions $ X:[s,t]\to H$; $L^2_{\dbF}(\Omega;L^1(s,t;H))$ is the set of $\dbF$-adapted, measurable processes $X : [s,t]\times\Omega\rightarrow
H$ such that $\dbE\Big[ \int_s^t|X(r)|dr\Big]^2<\infty$; $L^2_{\dbF}(s,t;H) $ is the set of $\dbF$-adapted, measurable processes $X : [s,t]\times\Omega\rightarrow
H$ such that $\dbE  \int_s^t|X(r)|^2dr  <\infty$; $L^2_{\dbF}(\Omega;C([s,t];H))$ is the set of $\dbF$-adapted, measurable, continuous processes $X : [s,t]\times\Omega\rightarrow
H$ such that $ \dbE\sup\limits_{r\in[s,t]}|X(r)|^2<\infty$; $C_{\dbF}([s,t];L^2(\Omega;\dbR^n))$ is the set of $\dbF$-adapted, measurable process $X:[s,t]\times\Omega\rightarrow H$ such that $r\mapsto X(\cd,r)$ is continuous and $\sup\limits_{r\in[s,t]}\dbE|X(r)|^2<\infty$.

In the following, let $K$ be a generic constant which varies in different context, and
\bel{Simplified-notations-1}\ba{ll}
\ns\ds \sA:= A+\wt A, \ \ \sB:= B+\wt B,\ \ \sC:= C+\wt C,\ \ \sD:= D+\wt D,\\
\ns\ds \sR:= R+\wt R,\ \ \sQ:=Q+\wt Q, \ \ \sG:=G+\wt G.
\ea\ee

\ms

If the state equation is a particular controlled SDE, we can use similar form of SDE to describe the $equilibrium$ $state$ $process$.  However, since the increment of state process $X$ in (\ref{standard-state-equation}) has the reliance on additional time reference $t$, the value of $X$ at time $s>t$ also depends on $t$. As a result, we need to propose an alternative kind of process as the equilibrium state process.

To get some inspirations from existing papers, we first look at one linear quadratic problem associated with state equation
\bel{Deterministic-ODE-1}\left\{\ba{ll}
\ns\ds dX(s)=\big[A(t,s)X(s)+B(t,s)u(s)\big]ds, \ \ s\in[t,T],\\
\ns\ds  X(t)=x,
\ea\right.
\ee
and cost functional
$$\ba{ll}
\ns\ds J(u(\cd);t,x)=\dbE_t\int_t^T \big[\lan Q(t,s)X(s),X(s)\ran +\lan R(t,s)u(s),u(s)\ran \big]ds+\dbE_t\lan G(t)X(T),X(T)\ran.
\ea
$$
Here the increment of state variable in (\ref{Deterministic-ODE-1}) also relies on initial time $t$, and discounting functions $Q,$ $R$ are not necessary to be exponential form. Both facts naturally lead to the time inconsistency of optimal control. According to \cite{Yong-2011-MCRF}, \cite{Yong-2010-math-sinica}, the equilibrium control $\bar u(\cd)$ not only relies on $\big\{s\in[0,T], Q(s,s), \ R(s,s)\big\}$, but also on equilibrium state $\bar X(\cd)$ define by
\bel{equilibrium-state-example}\left\{\ba{ll}
\ns\ds d\bar X(s)=\big[A(s,s)\bar X(s)+B(s,s)\bar u(s)\big]ds, \ \ s\in[t,T],\\
\ns\ds  \bar X(t)=x.
\ea\right.
\ee
In other words, both equilibrium control $\bar u$ and the equilibrium state $\bar X$ depends on the diagonal value (i.e., $t=s$) of coefficients $Q, R, A, B$.
Similar phenomenon also happens in investment and consumption problems with power-type utilities and general non-exponential discounting, see Section 6.2 of \cite{Yong-2012-2}.

We return back to our state equation (\ref{standard-state-equation}) again. Following the same principle as above (\ref{equilibrium-state-example}), it is expected that the corresponding equilibrium state process, denoted by $\sX^*(\cd)$, should satisfy %
\bel{Equilibrium-state-equation-general}\left\{\2n\ba{ll}
\ns\ds
d\sX^* =\Big[\sA\sX^*+\sB u^*+b\Big]ds +\Big[ \sC \sX^*+\sD u^*+\si\Big]dW(s),\q
s\in[0,T],\\
\ns\ds \sX^*(0)=x.\ea\right.\ee
with notations in (\ref{Simplified-notations-1}) and the equilibrium control $u^*(\cd)$. Keeping above arguments in mind, we introduce the following notion.

\ms

\begin{definition}\label{Definition-1}
Given initial state $x\in\dbR^n$, process $ u^*(\cd)\in L^2_{\dbF}(0,T;\dbR^m)$
is called an {\it open-loop equilibrium control} if for any
$t\in[0,T)$, small $\e>0$, $v\in L^2_{\cF_t}(\Omega;\dbR^m)$,
\bel{optimal-open}\lim_{\overline{\e\to0}}
{J(u^{v,\e}(\cd);t,\sX^*(t))-J\big(u^*(\cd)\big|_{[t,T]};t,\sX^*(t)\big)
\over\e}\ge0,\ee
where $\sX^*(\cd)$ satisfies (\ref{Equilibrium-state-equation-general}), and
$u^{v,\e}(s)=u^*(s)+vI_{[t,t+\e)}(s)$.
\end{definition}
If there is no mean-field term in (\ref{standard-state-equation}), then $X(\cd)$ only depends on $(x,u(\cd))$ and $\sX\equiv X$. Moreover, our definition reduces to the one in e.g., \cite{Hu-Jin-Zhou-2012}, \cite{Hu-Jin-Zhou-2017}.

We also introduce the closed-loop representation of open-loop equilibrium control $u^*(\cd)$.
\begin{definition}\label{Definition-2}
An open-loop equilibrium control $u^*(\cd)$ associated with initial state $x\in\dbR^n$ is said to have a closed-loop representation if $u^*(\cd)=\Th^*(\cd)\sX^*(\cd)+\f^*(\cd)$ where $(\Th^*(\cd),\f^*(\cd))\in L^2(0,T;\dbR^{m\times n})\times L^2_{\dbF}(0,T;\dbR^{m})$ is independent of $x$, $\sX^*(\cd)$ is the solution of
\bel{Closed-loop-representation-state-equation}\left\{\2n\ba{ll}
\ns\ds
d\sX^* =\Big[(\sA +\sB \Th^* )\sX^* +\sB \f^*  +b \Big]ds +\Big[(\sC +\sD \Th^* )\sX^* +\sD  \f^*  +\si \Big]dW(s),\q
s\in[0,T],\\
\ns\ds \sX^*(0)=x.\ea\right.\ee
\end{definition}
We emphasize that the structure of equilibrium state equation (\ref{Closed-loop-representation-state-equation}) is the same as that in \cite{Yong-2017} where the closed-loop equilibrium control of Problem (LQ) was formulated and investigated via multi-person differential games ideas.

\section{Characterizations of equilibrium controls}

In this section, we give the characterizations and explicit representations of open-loop equilibrium controls in the sense of Definition \ref{Definition-1}, \ref{Definition-2}.

\subsection{Some useful lemmas}

Given $u(\cd)\in L^2_{\dbF}(0,T;\dbR^m)$, $t\in[0,T]$, for later convenience we rewrite the state equation as follows
\bel{1.1}\left\{\2n\ba{ll}
\ns\ds
dX=\big[A  X  +\wt A  \dbE_tX +B u  +\wt B \dbE_tu +b \big]ds +\big[CX+\wt C  \dbE_tX %
+D u +\wt D  \dbE_tu  +\si \big]dW(s),\ \  s\in[0,T],\\
\ns\ds X(0)=x.\ea\right.\ee
We introduce the following BSDEs to deal with the quadratic cost functional (\ref{cost-functional-1}),
\bel{MF-BSDEs-*-1}\left\{ \ba{ll}
\ns\ds dY  =
- \Big[A ^{\top}Y  +\wt A ^{\top}\dbE_t Y  + C ^{\top} Z  + \wt C ^{\top} \dbE_t Z  -Q X -\wt Q\dbE_t X\Big]ds+Z  dW(s),\ \ s\geq t,\\
\ns\ds Y(T,t)=-G X (T)-\wt G \dbE_tX (T)-\g_2,\\
\ns\ds dY^{v,\e}_1 \!=-\Big[\!  A  ^{\top} Y^{v,\e}_1  +\wt A  ^{\top} \dbE_t Y^{v,\e}_1 \!+\wt C  ^{\top}\dbE_t Z^{v,\e}_1  + C ^{\top}
Z^{v,\e}_1 -\frac 1 2 Q X^{v,\e}_1 -\frac 1 2 \wt Q \dbE_tX^{v,\e}_1 \Big]ds +Z^{v,\e}_1 dW(s),\\
\ns\ds  Y^{v,\e}_1(T,t)=\!-\frac 1 2 \big[GX^{v,\e}_1(T)+\wt
G\dbE_tX^{v,\e}_1(T)\big],\!
\ea\right.\ee
where for $s\in[t,T]$, $X^{v,\e}_1(s):=X^{v,\e}_0(s)-X(s)$ satisfies
\bel{3.8} \ba{ll}
\left\{\2n\ba{ll}
\ns\ds
dX^{v,\e}_1 =\[A X^{v,\e}_1 +\wt A \dbE_t X^{v,\e}_1 +B \big[u^{v,\e}_0 -u \big]
+\wt B \dbE_t\big[u^{v,\e}_0 -u  \big]  \]ds\\
\ns\ds\qq\qq\ +\[C X^{v,\e}_1 +\wt C \dbE_tX^{v,\e}_1  +D \big[u^{v,\e}_0 -u \big]
+\wt D(s)\dbE_t\big[u^{v,\e}_0(s)-u (s)\big] \]dW(s),
\\
\ns\ds X^{v,\e}_1(t)=0,
\ea\right.
\ea\ee
and
$u^{v,\e}_0(\cd):=u(\cd)+vI_{[t,t+\e)}(\cd),$ $v\in L^2_{\cF_t}(\Omega;\dbR^m).$
We also define $Y_0(\cd)$ as
\bel{Definition-of-Y-0}\left\{\ba{ll}
\ns\ds dY_0 =- \Big[A ^{\top}+\wt A ^{\top}\Big] Y_0 ds,\ \ s\in[0,T],\\
\ns\ds Y_0(T)=-I_{n\times n}.
\ea\right.\ee
\begin{remark}
As to $X^{v,\e}_1(\cd)$, by some standard calculations one has
\bel{Integrability-for-X-1}\ba{ll}
\ns\ds
\dbE_t\Big[\sup\limits_{_{t\in[t,t+\e]}}|X^{v,\e}_1(s)|^p\Big] \leq
K \[\int_t^{t+\e}|\sB(r)v|dr\]^p+K \Big[\int_t^{t+\e}| \sD(r)v|^2dr\Big]^{\frac p 2},\ \ a.s.\\
\ns\ds \dbE_t\Big[\sup\limits_{_{t\in[t+\e,T]}}|X^{v,\e}_1(s)|^p\Big] \leq K\dbE_t |X_1^{v,\e}(t+\e)|^p, \ \ a.s.
\ea\ee
where $K$ only depends on $p>1$ and $v\in L^2_{\cF_t} (\Omega;\dbR^m)$. Given the backward equations in (\ref{MF-BSDEs-*-1}), for any $t\in[0,T)$ and small $\e>0$, we see that they are uniquely solvable with
\bel{Integral-conditions-for-Y-1-Y-2}\ba{ll}
\ns\ds (Y^{v,\e}_1(\cd,t),Z^{v,\e}_1(\cd,t)), (Y (\cd,t),Z (\cd,t))\in L^2_{\dbF}(\Omega;C([t,T];\dbR^n))\times L^2_{\dbF}(t,T;\dbR^n).
\ea\ee
\end{remark}
\br
When there is no mean-field terms in (\ref{standard-state-equation}), the pair of processes $(Y^{v,\e}_1,Z^{v,\e}_1)$ of (\ref{MF-BSDEs-*-1}) appeared in \cite{Wang-2017-JDE}, but were absent in \cite{Hu-Jin-Zhou-2012}, \cite{Hu-Jin-Zhou-2017} and \cite{Yong-2017}.
\er
The following result shows the roles of previous $(Y ,Z )$, $(Y_1^{v,\e},Z^{v,\e}_1)$.
\begin{lemma}\label{first-preparing-lemma}
Given the notations in (\ref{Simplified-notations-1}), $\e>0$, $u (\cd)\in L^2_{\dbF}(0,T;\dbR^m)$, suppose $(Y ,Z )$, $(Y_1^{v,\e} ,Z_1^{v,\e} )$ satisfy (\ref{Integral-conditions-for-Y-1-Y-2}) and (\ref{MF-BSDEs-*-1}). Then for any $t\in[0,T]$,
$$\ba{ll}
\ns\ds J(u^{v,\e}_0(\cd);t,X(t))-J\big(u(\cd);t,X(t)\big)\\
\ns\ds\qq =\dbE_t\1n\int_t^{t+\e}\1n\lan
 \sR(s)\big[\frac{v\1n}{2}+u (s)\big]- \sB(s)^{\top}\big[ Y_0(s)\g_1 X(t)+Y^{v,\e}_1(s,t)+Y (s,t)\big]\\
\ns\ds\qq\qq
- \sD(s)^{\top}(Z^{v,\e}_1(s,t)+Z (s,t)),v\ran ds.
\ea$$
\end{lemma}
\begin{proof}
For any $t\in[0,T)$ and small $\e$ such that $t+\e\leq T$, we slip $[t,T]$ into $[t,t+\e)$
and $[t+\e,T]$. If $s\in[t,t+\e)$, we have
$$\ba{ll}
\ns\ds\dbE_t\int_t^{t+\e}\[\lan
Q(s)X^{v,\e}_0(s),X^{v,\e}_0(s)\ran-\lan Q(s)X (s),X (s)\ran
\\
\ns\ds \qq\qq+\lan R(s)u^{v,\e}_0(s),u^{v,\e}_0(s)\ran-\lan R(s)u (s),u (s)\ran\]ds\\
\ns\ds=\dbE_t\1n\int_t^{t+\e}\1n\[\lan Q(s)\big[X^{v,\e}_1(s)+2
X (s)\big],X^{v,\e}_1(s)\ran+\lan R(s)\big[v+2u (s)\big],v\ran\]ds,
\ea$$
where $X^{v,\e}_1(\cd)$ is defined in (\ref{3.8}). Similarly
$$\ba{ll}
\ns\ds \dbE_t\int_t^{t+\e}\[\lan
\wt Q(s)\dbE_tX^{v,\e}_0(s),\dbE_tX^{v,\e}_0(s)\ran-\lan \wt Q(s)\dbE_tX (s),\dbE_tX (s)\ran\\
\ns\ds \qq\qq+\lan \wt R(s)\dbE_tu^{v,\e}_0(s),\dbE_tu^{v,\e}_0(s)\ran-\lan \wt R(s)\dbE_tu (s),\dbE_tu (s)\ran\]ds\\
\ns\ds=\dbE_t\1n\int_t^{t+\e}\1n\[\lan \wt Q(s)\big[\dbE_tX^{v,\e}_1(s)+2
\dbE_tX (s)\big],\dbE_tX^{v,\e}_1(s)\ran+\lan \wt R(s)\big[v+2\dbE_tu (s)\big],v\ran\]ds.
\ea
$$
If $s\in[t+\e,T]$,
$$\ba{ll}
\ns\ds\dbE_t\int_{t+\e}^T\[\lan
Q(s)X^{v,\e}_0(s),X^{v,\e}_0(s)\ran-\lan Q(s)X (s),X (s)\ran
\\
\ns\ds \qq\qq+\lan R(s)u^{v,\e}_0(s),u^{v,\e}_0(s)\ran-\lan R(s)u (s),u (s)\ran\]ds\\
\ns\ds=\dbE_t\1n\int_{t+\e}^T \lan Q(s) \big[X^{v,\e}_1(s)+2
X (s)\big],X^{v,\e}_1(s)\ran ds.
\ea$$
Similarly we have
$$\ba{ll}
\ns\ds\dbE_t\int_{t+\e}^T\[\lan \wt
Q(s)\dbE_tX^{v,\e}_0(s),\dbE_tX^{v,\e}_0(s)\ran-\lan \wt Q(s)\dbE_tX (s),\dbE_tX (s)\ran
\\
\ns\ds \qq\qq+\lan \wt R(s)\dbE_tu^{v,\e}_0(s),\dbE_tu^{v,\e}_0(s)\ran-\lan \wt R(s)\dbE_tu (s),\dbE_tu (s)\ran\]ds\\
\ns\ds=\dbE_t\1n\int_{t+\e}^T \lan \wt Q(s) \big[\dbE_tX^{v,\e}_1(s)+2
\dbE_tX (s)\big],\dbE_tX^{v,\e}_1(s)\ran ds.
\ea$$
To sum up, for any $t\in[0,T)$ we obtain that,
\bel{3.11} \ba{ll}
\ns\ds
J(u^{v,\e}_0(\cd);t,X (t))-J\big(u (\cd);t,X (t)\big)=J_1(t )+J_2(t)
\ea\ee
where
$$\left\{\ba{ll}
\ns\ds J_1(t ):=\frac 1 2 \dbE_t\1n\int_t^{T}\1n \Big[\lan
Q(s)X^{v,\e}_1(s),X^{v,\e}_1(s)\ran +\lan
\wt Q(s)\dbE_tX^{v,\e}_1(s),\dbE_tX^{v,\e}_1(s)\ran\Big]
ds
\\
\ns\ds\qq\qq+\frac 1 2 \int_t^{t+\e}\lan \sR(s) v,v\ran ds+\frac 1 2 \dbE_t\lan GX^{v,\e}_1(T)+\wt
G\dbE_tX^{v,\e}_1(T),X^{v,\e}_1(T)\ran,\\
\ns\ds  J_2(t ):=\dbE_t\1n\int_t^{T} \Big[\lan Q(s)
X(s),X^{v,\e}_1(s)\ran +\lan \wt Q(s)
\dbE_tX(s),\dbE_tX^{v,\e}_1(s)\ran \Big] ds\\
\ns\ds\qq\qq+ \int_t^{t+\e}\lan \sR(s)\dbE_tu (s),v\ran ds +\dbE_t\lan GX (T)+\wt
G\dbE_tX (T)+\gamma_1x+\gamma_2,X^{v,\e}_1(T)\ran.
\ea\right.$$
Since $(Y (\cd,t),Z (\cd,t))$ satisfies (\ref{MF-BSDEs-*-1}), by It\^o's formula to $\lan
Y (\cd,t),X^{v,\e}_1(\cd)\ran$ and the fact $X^{v,\e}_1(t)=0$,
\bel{3.12} \ba{ll}
\ns\ds\dbE_t\lan Y (T,t),X^{v,\e}_1(T)\ran\\
\ns\ds  = \dbE_t\1n\int_t^{T}\lan Q(s) X (s),X^{v,\e}_1(s)\ran ds+ \1n\int_t^{T}\lan \wt Q(s) \dbE_tX (s),\dbE_tX^{v,\e}_1(s)\ran ds\\
\ns\ds\qq +\dbE_t\int_t^{t+\e}\lan \sB(s)^{\top}Y (s,t)+ \sD(s)^{\top}Z (s,t),v
\ran ds.
 \ea\ee
Similarly we also obtain that,
\bel{3.13} \ba{ll}
\ns\ds\dbE_t\lan Y^{v,\e}_1(T,t),X^{v,\e}_1(T)\ran\\
\ns\ds=\frac 1 2 \dbE_t\1n\int_t^{T}\lan Q(s) X^{v,\e}_1(s),X^{v,\e}_1(s)\ran ds+\frac 1 2 \dbE_t\1n\int_t^{T}\lan \wt Q(s) \dbE_tX^{v,\e}_1(s),\dbE_tX^{v,\e}_1(s)\ran ds\\
\ns\ds\qq
  +\dbE_t\int_t^{t+\e}\lan
 \sB(s)^{\top}Y^{v,\e}_1(s,t)+ \sD(s)^{\top}Z^{v,\e}_1(s,t),v \ran ds.
\ea\ee
In order to deal with $\lan \g_1X(t),\dbE_t X_1^{v,\e}(T)\ran$, we use It\^{o}'s formula to $\lan \g_1X(t) Y_0(\cd),X_1^{v,\e}(\cd)\ran$ on $[t,T]$,
\bel{gamma-1-x-ito}\ba{ll}
\ns\ds \dbE_t \lan \g_1X(t),X_1^{v,\e}(T)\ran =-\dbE_t\int_t^{t+\e}\lan  \sB(s)^{\top}  Y_0(s)\g_1X(t),v\ran ds.
\ea\ee
Therefore, our conclusion follows from (\ref{3.11}), (\ref{3.12}), (\ref{3.13}), (\ref{gamma-1-x-ito}).
\end{proof}
In the following, we first deal with the terms with respect to $(Y^{v,\e}_1(\cd,\cd),Z^{v,\e}_1(\cd,\cd))$ in Lemma \ref{first-preparing-lemma}.
\begin{lemma}\label{Second-preparing-lemma}
Suppose (H1) holds, $(X^{v,\e}_1,Y^{v,\e}_1,Z^{v,\e}_1)$ satisfy (\ref{MF-BSDEs-*-1}) and (\ref{3.8}). Then for almost $t\in[0,T]$,
\bel{3.15}\ba{ll}
\ns\ds\frac{1}{\e}\dbE_t\int_t^{t+\e}\lan
 \sB(s)^{\top}Y^{v,\e}_1(s,t)+ \sD (s)^{\top}Z^{v,\e}_1(s,t),v\ran
ds =\frac{1}{2\e}\dbE_t \lan \int_t^{t+\e}  \sD (s)^{\top}\h P_1(s) \sD(s)ds
v,v\ran+o(1),
\ea\ee
where $\sB(\cd),\ \sD(\cd)$ are defined in (\ref{Simplified-notations-1}), $\h P_1(\cd)$ satisfies
\bel{Equation-for-hat-P-1}\left\{\ba{ll}
\ns\ds d\h P_1=-\Big[\h P_1 A+A^{\top}\h P_1+C^{\top}\h P_1C-Q\Big]ds,\ \ s\in[0,T],\\
\ns\ds \h P_1(T)=-G.
\ea\right.\ee
\end{lemma}
\begin{proof}
For any $t\in[0,T)$ and small $\e>0$ such that $t+\e\leq T$, from the definition of $u^{v,\e}_0(\cd)$ we separate $[t,T]$ into $[t+\e,T]$ and $[t,t+\e)$. Firstly we look
at the case on $[t+\e,T]$, where
\bel{MF-FBSDE-v-e-1} \ba{ll}
\left\{\2n\ba{ll}
\ns\ds
dX^{v,\e}_1(s)=\big[ A(s)X^{v,\e}_1(s)+\wt A(s)\dbE_t X^{v,\e}_1(s) \big] ds + \big[C(s)X^{v,\e}_1(s)+\wt C(s)\dbE_tX^{v,\e}_1(s)\big] dW(s),
\\
\ns\ds X^{v,\e}_1(t+\e)=X^{v,\e}_1(t+\e).\ea\right.
\ea\ee
To represent $(Y_1^{v,\e},Z^{v,\e}_1)$ in $[t+\e,T]$, we need above $\h P_1(\cd)$ and the following $\h P_2(\cd)$ on $[0,T]$,
$$\left\{
\ba{ll}
\ns\ds d\h P_2=-\Big[\h P_2\sA+\sA^{\top}\h P_2+\h P_1\wt A
+\wt A^{\top}\h P_1+\wt C^{\top}\h P_1\sC +C^{\top}\h P_1\wt C-\wt Q\Big]ds,\\
\ns\ds \ \h P_2(T)=-\wt G.
\ea\right.$$
We define two processes as,
\bel{ }\left\{\ba{ll}
\ns\ds  \cY^{v,\e}_1(s,t):=\frac  1 2 \h P_1(s)X^{v,\e}_1(s)+\frac  1 2 \h P_2(s)\dbE_t[X^{v,\e}_1(s)],\qq  s\in[t+\e,T], \\
\ns\ds  \cZ_1^{v,\e}(s):=\frac 1 2\h P_1(s)\big[ C(s)X^{v,\e}_1(s)+\wt C(s)\dbE_t X^{v,\e}_1(s)\big],\ \ s\in[0,T].
\ea\right.
\ee
We see that $(\cY^{v,\e}_1, \cZ^{v,\e}_1)\in L^2_{\dbF}(\Omega;C([t+\e,T];\dbR^n))\times L^2_{\dbF}(0,T;\dbR^n)$. For $t\in[0,T),$ since
$$\ba{ll}
 \ns\ds
 d\dbE_t X_1^{v,\e}(s)= \sA(s) \dbE_t X_1^{v,\e}(s)ds,\ \ s\in[t,T],
\ea
$$
it then follows from It\^{o}'s formula that
$$\left\{
\ba{ll}
\ns\ds d\big[ \h P_1X_1^{v,\e}\big]=\Big[-\big(A^{\top}\h P_1+C^{\top}\h P_1C-Q\big) X_1^{v,\e}+\h P_1 \wt A\dbE_t X^{v,\e}_1\Big]ds+\h P_1(CX^{v,\e}_1+\wt C\dbE_t X^{v,\e}_1)dW(s),\\
\ns\ds d\Big[\h P_2\dbE_t\big(X^{v,\e}_1\big)\Big]=-\Big\{\sA^{\top}\h P_2+\h P_1\wt A+\wt A^{\top}\h P_1+\wt C^{\top}\h P_1\sC +C^{\top}\h P_1\wt C-\wt Q
\Big\}\dbE_t\big(X^{v,\e}_1\big)ds.
\ea\right.
$$
As a result, for $s\in[t,T]$ we see that
$$
\ba{ll}
\ns\ds d \cY_1^{v,\e}= -\frac 1 2 \Big\{\sA ^{\top}\h P_2+\wt A^{\top}\h P_1+\wt C^{\top}\h P_1\sC +C^{\top}\h P_1\wt C-\wt Q
\Big\}\dbE_t\big(X^{v,\e}_1\big)ds\\
\ns\ds\qq\qq\q\
- \frac 1 2
\big(A^{\top}\h P_1+C^{\top}\h P_1C-Q\big)X^{v,\e}_1 ds
 +\frac 1 2 \h P_1(CX^{v,\e}_1+\wt C\dbE_t X^{v,\e}_1)dW(s).
\ea
$$
On the other hand,
$$
\ba{ll}
\ns\ds -\wt A^{\top}\dbE_t \cY_1^{v,\e}-\wt C^{\top}\dbE_t \cZ_1^{v,\e}-
A^{\top} \cY_1^{v,\e}-C^{\top}\cZ_1^{v,\e}\\
\ns\ds= -\frac 1 2 \Big[\wt A^{\top}(\h P_1+\h P_2)+\wt C^{\top}\h P_1 \sC +A^{\top}\h P_2+C^{\top}\h P_1\wt C\Big]\dbE_t X_1^{v,\e}-\frac 1 2 \Big[A^{\top}\h P_1 +C^{\top}\h P_1 C\Big]X^{v,\e}_1.
\ea
$$
Hence $(\cY^{v,\e}_1, \cZ^{v,\e}_1)$ satisfies the second backward equation of (\ref{MF-BSDEs-*-1}). The uniqueness of BSDEs show that
$$\ba{ll}
\ns\ds   \dbP\Big\{\omega\in\Omega;\ \cY^{v,\e}_1(s,t)=Y^{v,\e}_1(s,t),\ \ \forall s\in[t+\e,T]\Big\}=1,\ \ t\in[0,T],\\
\ns\ds \dbP\Big\{\omega\in\Omega;\ \cZ^{v,\e}_1(s,t) = Z^{v,\e}_1(s)\Big\}=1,\ \ s\in[t+\e,T], \ \ a.e. \ \ t\in[0,T].
\ea
$$
By the value of $Y^{v,\e}_1(t+\e,t)$, we
continue to study
$(X^{v,\e}_1(\cd),Y^{v,\e}_1(\cd,t),Z^{v,\e}_1(\cd,t))$ on
$[t,t+\e]$,
\bel{MF-FBSDE-v-e-2} \ba{ll}
\left\{\2n\ba{ll}
\ns\ds
dX^{v,\e}_1=\[A X^{v,\e}_1 +\wt A \dbE_tX^{v,\e}_1 +\sB  v\]ds   +\[C X^{v,\e}_1+
\wt C \dbE_t X^{v,\e}_1  +\sD v\]dW(s),\ \
\\
\ns\ds dY^{v,\e}_1 =-\Big[A ^{\top}Y^{v,\e}_1 +\wt A ^{\top}\dbE_t Y^{v,\e}_1 +C ^{\top}Z^{v,\e}_1 +\wt C ^{\top}\dbE_tZ^{v,\e}_1   -\frac 1 2 Q X_1^{v,\e} -\frac 1 2 \wt Q \dbE_t X_1^{v,\e} \Big]ds+Z_1^{v,\e} dW(s), \\
\ns\ds X^{v,\e}_1(t)=0,\ \ Y^{v,\e}_1(t+\e,t)=Y^{v,\e}_1(t+\e,t). \ea\right.
\ea\ee
To represent above $(Y_1^{v,\e},Z_1^{v,\e})$, we define two processes on $[t,t+\e]$,
\bel{ }\left\{\ba{ll}
\ns\ds  \h Y^{v,\e}_1(s,t):=\frac  1 2 \h P_1(s)X^{v,\e}_1(s)+\frac  1 2 \h P_2(s)\dbE_t[X^{v,\e}_1(s)]+\frac 1 2 \h P_3(s)v, \\
\ns\ds \h Z_1^{v,\e}(s):=\frac 1 2\h P_1(s)\big[C(s)X^{v,\e}_1(s)+\wt C(s)\dbE_t X^{v,\e}_1(s) +\sD (s)v\big],
\ea\right.
\ee
where $\h P_1(\cd)$, $\h P_2(\cd)$ appeared above and for $s\in[t,t+\e],$
$$
\left\{ \ba{ll}
\ns\ds d\h P_3=-\big[\sA^{\top} P_3+\sC^{\top}\h P_1\sD+(\h P_1+\h P_2)\sB\big]ds,\\
\ns\ds \h P_3(t+\e)=0.
\ea\right.
$$
%
Here we observe that
$$
\ba{ll}
\ns\ds  d\dbE_t X_1^{v,\e}=\[\sA \dbE_t X_1^{v,\e}+\sB v\]ds,\ \ s\in[t,t+\e].
\ea
$$
Using It\^{o}'s formula, we see that
$$\left\{
\ba{ll}
\ns\ds d\big[ \h P_1X_1^{v,\e}\big]=\Big[-(A^{\top}\h P_1  +C^{\top}\h P_1C -Q)X_1^{v,\e}+\h P_1\wt A\dbE_tX_1^{v,\e}+\h P_1 \sB v \Big]ds\\
\ns\ds\qq\qq\q
+\h P_1\big(CX^{v,\e}_1+\wt C\dbE_t X^{v,\e}_1+\sD v\big)dW(s),\\
\ns\ds d\Big[\h P_2\dbE_t\big(X^{v,\e}_1\big)\Big]=
\Big\{-\Big[\sA^{\top}\h P_2+\h P_1\wt A+\wt A^{\top}\h P_1+\wt C^{\top}\h P_1 \sC+C^{\top}\h P_1\wt C-\wt Q
\Big]\dbE_t\big(X^{v,\e}_1\big) +\h P_2\sB v\Big\}ds.
\ea\right.
$$
On the other hand, according to the definition of $(\h Y_1^{v,\e},\h Z_1^{v,\e})$, we have
$$
\ba{ll}
\ns\ds  -
A ^{\top} \h Y_1^{v,\e} -C ^{\top}\h Z_1^{v,\e}  -
\wt A ^{\top} \dbE_t\h Y_1^{v,\e} -\wt C ^{\top}\dbE_t\h Z_1^{v,\e}+\frac 1 2 Q X_1^{v,\e}+\frac 1 2 \wt Q\dbE_t X_1^{v,\e}\\
\ns\ds = -\frac 1 2 \Big[\wt A^{\top}(\h P_1+\h P_2)+\wt C^{\top}\h P_1\sC+A^{\top}\h P_2+C^{\top}\h P_1\wt C- \wt Q\Big]\dbE_t X_1^{v,\e}\\
\ns\ds\q-\frac 1 2 \Big[A^{\top}\h P_1 +C^{\top}\h P_1 C- Q\Big]X^{v,\e}_1-\frac 1 2 \sA ^{\top}P_3 v-\frac 1 2 \sC^{\top}\h P_1 \sD v,
\ea
$$
from which $(\h Y_1^{v,\e},\h Z_1^{v,\e})$ satisfies BSDE (\ref{MF-BSDEs-*-1}). Moreover, it is easy to see
\bel{}\ba{ll}
\ns\ds (\h Y^{v,\e}_1(\cd,t),\h Z^{v,\e}_1(\cd,t))\in L^2_{\dbF}(\Omega;C([t,t+\e];\dbR^n))\times L^2_{\dbF}(t,t+\e;\dbR^n).
\ea\ee
The uniqueness of BSDEs shows that $(\h Y^{v,\e}_1(\cd,t),\h Z^{v,\e}_1(\cd,t))=( Y^{v,\e}_1(\cd,t), Z^{v,\e}_1(\cd,t))$ on $[t,t+\e]$.
As a result, for any $s\in[t,t+\e),$
\bel{B-Y-1-D-Z-1}\ba{ll}
\ns\ds \sB ^{\top}Y_1^{v,\e}+\sD  ^{\top}Z^{v,\e}_1
=\frac 1 2 \Big\{\big[\sB ^{\top}\h P_1+\sD^{\top}\h P_1C\big]X_1^{v,\e}+\big[\sB ^{\top}\h P_2+\sD ^{\top}\h P_1\wt C\big]\dbE_tX_1^{v,\e} +\big[\sB ^{\top}\h P_3+\sD ^{\top}\h P_1\sD \big]v\Big\}.
\ea\ee
By (\ref{Integrability-for-X-1}), we know that $\dbE_t\Big[\sup\limits_{_{t\in[t,t+\e)}}|X^{v,\e}_1(s)|^2\Big]
=O(\e)$, from which we derive
\bel{ } \left\{\2n\ba{ll}
\ns\ds\frac{1}{\e}\dbE_t\int_t^{t+\e}\lan
\[ \sB(s) ^{\top}\h P_1(s)+ \sD (s) ^{\top}\h P_1(s)C(s)
\]X^{v,\e}_1(s),v\ran ds=o(1),\\
\ns\ds\frac{1}{\e}\dbE_t\int_t^{t+\e}\lan \big[\sB (s) ^{\top} P_2(s)+\sD ^{\top}\h P_1\wt C\big]\dbE_t[X^{v,\e}_1(s)]+\sB(s) ^{\top}\h P_3(s)v,v\ran
ds=o(1).\ea\right. \ee
Then our conclusion is easy to see.
\end{proof}
\br
Notice that (\ref{Equation-for-hat-P-1}) is consistent with the second-order adjoint equation in optimal control problem of mean-field SDEs, see e.g., \cite{Buckdahn-Djehiche-Li-2011}.
\er

\subsection{The case of open-loop equilibrium controls}

In this part, we give the characterizations of open-loop equilibrium controls.

At first, we introduce a representation for $(Y ,Z )$ in (\ref{MF-BSDEs-*-1}).  %
We observe that
$$
d\dbE_t X=\Big[\sA\dbE_t X+\sB\dbE_t\f+\dbE_t b\Big]ds.
$$
For $t\in[0,T]$ and $s\in[t,T]$, suppose that
\bel{Representation-for-Y-*}\ba{ll}
\ns\ds  Y(s,t)=\cP_1(s)X(s)+  \cP_2(s)\dbE_tX(s)+\dbE_t \cP_{3}(s)+ \cP_{4}(s),
\ea
\ee
where $\cP_1(\cd),\ \cP_2(\cd)$ are deterministic, $\cP_3(\cd), \ \cP_4(\cd)$ are stochastic processes satisfying
$$
\ba{ll}
\ns\ds  d \cP_i(s)=\Pi_{i}(s)ds,\ \ i=1,2, \ \ \cP_1(T)=-G,\ \  \cP_2(T)=-\wt G, \\
\ns\ds d \cP_j(s)=\Pi_j(s)ds+\cL_j(s)dW(s),\ \ j=3,4,\ \  \cP_3(T)=0,\ \ \cP_{4}(T)=-\gamma_2.
\ea
$$
Here $\Pi_i(\cd)$ are to be determined. Using It\^{o}'s formula, we derive that
$$\left\{
\ba{ll}
\ns\ds d\big[  \cP_1 X\big]=\Big[\Pi_{1} X+ \cP_1(A X +\wt A\dbE_tX +B u +\wt B\dbE_t u+b)\Big]ds\\
\ns\ds\qq\qq\q+ \cP_1\big( C X+D u+ \wt C \dbE_tX+\wt D\dbE_t u+\si\big)dW(s),\\
\ns\ds d\big[\cP_2\dbE_t X\big]=\Big\{\Pi_2 \dbE_t X+\cP_2 \big[\sA \dbE_tX+\sB\dbE_t u+\dbE_tb\big]\Big\}ds.
\ea\right.
$$
As a result, we have
$$
\ba{ll}
\ns\ds d Y=\Big\{\Pi_{1} X+ \cP_1( A X +\wt A\dbE_tX +B u+\wt B\dbE_t u+b)+\Pi_2 \dbE_t X
\\
\ns\ds\qq\ +\cP_2 \big[ \sA \dbE_tX+\sB\dbE_t u+\dbE_tb\big]
+\dbE_t\Pi_3+\Pi_4 \Big\}ds\\
\ns\ds\qq\ +\Big[ \cP_1\big( C X+\wt C \dbE_tX +D u+\wt D\dbE_t u+\si\big)+\cL_{4}\Big]dW(s).
\ea
$$
Consequently, it is necessary to see
\bel{Representation-for-Z-*}
\ba{ll}
\ns\ds Z= \cP_1\big( C X+\wt C \dbE_tX +Du+\wt D\dbE_t u+\si\big)+\cL_{4}.
\ea
\ee
In this case, from (\ref{Representation-for-Y-*}), (\ref{Representation-for-Z-*}), we see that
$$\left\{
\ba{ll}
\ns\ds \dbE_t Y=(\cP_1+\cP_2)\dbE_t X +\dbE_t \big[\cP_{3}+\cP_{4}\big],\\
\ns\ds \dbE_t Z=\dbE_t\Big[\cP_1\big(\sC X +\sD u +\si\big)\Big]+\dbE_t\cL_{4}.
\ea\right.
$$
On the other hand, by the previous representations,
$$
\ba{ll}
\ns\ds -
A^{\top} Y-C^{\top}Z -
\wt A^{\top} Y-\wt C^{\top}Z+Q X+\wt Q\dbE_t X\\
\ns\ds=  -A^{\top}\Big\{\cP_1X+\cP_2\dbE_t X +\dbE_t \cP_{3}+ \cP_{4} \Big\}
-C^{\top} \Big\{\cP_1\big( C X+\wt C \dbE_tX +D u +\wt D\dbE_t u+\si\big)+\cL_{4}\Big\}\\
\ns\ds\q-\wt A^{\top}\Big[(\cP_1+\cP_2)\dbE_t X +\dbE_t \big[\cP_{3}+\cP_{4}\big]\Big] -\wt C^{\top}\Big[\dbE_t\Big[\cP_1\big(\sC X +\sD u +\si\big)\Big]+\dbE_t\cL_{4}\Big]+Q X+\wt Q\dbE_t X.
\ea
$$
At this moment, we can choose $\Pi_i(\cd)$ in the following ways,
$$\left\{
\ba{ll}
\ns\ds 0=\Pi_1+\cP_1 A+A^{\top}\cP_1+C^{\top}\cP_1 C-Q,\ \    \\
\ns\ds 0=\Pi_2+\cP_2\sA+\sA^{\top} \cP_2+\cP_1\wt A+\wt A^{\top} \cP_1
+\wt C^{\top}\cP_1\sC+C^{\top}\cP_1\wt C-\wt Q
\\
\ns\ds 0=\Pi_{4} +A^{\top}\cP_{4}+C^{\top}\cL_{4}+\cP_1(B u +b)+C^{\top}\cP_1(D u+\si),\\
\ns\ds 0=\Pi_{3}+\wt A^{\top}[ \cP_{3}+\cP_{4}]+A^{\top}\cP_{3}
+\wt C^{\top}\cL_4
+\cP_1\wt B u \\
\ns\ds\qq  +\cP_2\big[\sB u+b\big]+\wt C^{\top}\cP_1\big[\sD u+ \si\big]+C^{\top}\cP_1\wt D u.
\ea\right.
$$
Next we make above arguments rigorous.
Given the notations in (\ref{Simplified-notations-1}), for $s\in[0,T]$, we consider the following systems of equations
\bel{Equations-for-P-i-in-uniqueness}\left\{
\ba{ll}
\ns\ds d\cP_1=-(\cP_1 A+A^{\top}\cP_1+C^{\top}\cP_1 C-Q)ds,\ \ \\
\ns\ds d\cP_2=-\Big\{\cP_2\sA+\sA^{\top}\cP_2+\cP_1\wt A+\wt A^{\top} \cP_1
+\wt C^{\top}\cP_1\sC+C^{\top}\cP_1\wt C-\wt Q\Big\}ds,\\
\ns\ds d\cP_{3}=-\Big[ \sA^{\top} \cP_{3}+\wt A^{\top} \cP_{4}
+\wt C^{\top}\cL_4
+\cP_1\wt B u +\cP_2\big[\sB u+b\big]+\wt C^{\top}\cP_1\big[\sD u+ \si\big]\\
\ns\ds\qq\qq +C^{\top}\cP_1\wt D u\Big]ds+\cL_{3}dW(s),\\
d\cP_4=-\Big\{A^{\top}\cP_{4}+C^{\top}\cL_{4}+\cP_1(B u +b)+C^{\top}\cP_1(D u+\si)\Big\}ds+\cL_4dW(s),
\\
\ns\ds \cP_1(T)=-G,\ \cP_2(T)=-\wt G,\ \cP_3(T)=0,\ \cP_4(T)=-\gamma_2.
\ea\right.\ee
For $u(\cd)\in L^2_{\dbF}(0,T;\dbR^m)$, under (H1) it is obvious to see the existence and uniqueness of $\cP_1(\cd)$, $\cP_2(\cd),$ $(\cP_3(\cd),\cL_3(\cd)), \ (\cP_4(\cd),\cL_4(\cd))$ and
$$
\ba{ll}
\ns\ds \cP_1(\cd),\ \cP_2(\cd)\in C([0,T];\dbR^{n\times n}),\ (\cP_{3 },\cL_{3}),(\cP_{4},\cL_{4})\in  L^2_{\dbF}(\Omega;C([0,T];\dbR^{n}))\times L^2_{\dbF}(0,T;\dbR^{n}).
\ea
$$
At this moment, for $s\in[0,T]$, and $t\in[0,s]$, we define a pair of processes
\bel{Arbitrary-s-Y-Z-general}\left\{
\ba{ll}
\ns\ds \sY(s,t):=\cP_1(s)X(s)+\cP_2(s)\dbE_tX(s) + \dbE_t\cP_3(s)+\cP_4(s),  \\
\ns\ds \sZ(s,t):=\cP_1(s)\big( C(s) X(s)+\wt C(s)\dbE_t X(s)+D(s)u(s) +\wt D(s)\dbE_t u(s)+\si(s)\big)+\cL_4(s).
\ea\right.\ee
By the results of $\cP_i(\cd)$, we can conclude that
$$(\sY_{d}(\cd),\sZ_d(\cd))\in L^2_{\dbF}(\Omega;C([0,T];\dbR^n))\times L^2_{\dbF}(0,T;\dbR^n)$$
where $(\sY_{d}(s ),\sZ_{d}(s ))\equiv(\sY(s,s),\sZ(s,s))$ with $s\in[0,T]$.
We present the following representation for $(Y,Z)$.
\bl\label{representation-for-Y-Z-*}
\rm Suppose $u(\cd)\in L^2_{\dbF}(0,T;\dbR^m)$ and $(Y,Z)$ is the unique pair of solution for the first BSDE in (\ref{MF-BSDEs-*-1}). Then for any $t\in[0,T]$,
\bel{representation-for-any-y-z}\ba{ll}
\ns\ds   \dbP\Big\{\omega\in\Omega;\ Y(s,t)=\sY(s,t),\ \ \forall s\in[t,T] \Big\}=1,\\
\ns\ds \dbP\Big\{\omega\in\Omega;\ Z (s,t) = \sZ(s,t)\Big\}=1,\ \ s\in[t,T]. \ a.e.
\ea
\ee
\el
\begin{proof}
Given (\ref{Arbitrary-s-Y-Z-general}), it is easy to see that
$$\left\{
\ba{ll}
\ns\ds \dbE_t\sY =(\cP_1+\cP_2)\dbE_tX+  \dbE_t[\cP_3+\cP_4],\\
\ns\ds \dbE_t\sZ =P_1\Big[\sC\dbE_t X+ \sD  \dbE_t u+\dbE_t\si \Big]+\dbE_t\cL_4.
\ea\right.$$
Using It\^{o}'s formula, we know that
$$\left\{
\ba{ll}
\ns\ds d\big[\cP_1X \big]=-\big[A^{\top}\cP_1+C^{\top}\cP_1 C-Q\big]
X ds+P_1(\wt A\dbE_t X+ Bu +\wt B\dbE_tu +b)ds\\
\ns\ds\qq\qq\q+P_1\big[ C X+\wt C\dbE_t X+D u+\wt D\dbE_tu  +\si\big]dW(s),\\
\ns\ds  d\big[\cP_2\dbE_t X\big]=\Big\{\cP_2 \big[ \sB\dbE_t u+\dbE_tb\big]-\big[\sA^{\top} \cP_2+\cP_1\wt A+\wt A^{\top}\cP_1
+\wt C^{\top}\cP_1 \sC+C^{\top}\cP_1\wt C-\wt Q\big] \dbE_t X\Big\}ds.
\ea\right.$$
Consequently, after some calculations one has
$$
\ba{ll}
\ns\ds d \sY=d\Big[\cP_1X+\cP_2\dbE_t X +\dbE_t\cP_3+\cP_4 \Big]\\
\ns\ds \qq=\Big\{-(A^{\top}\cP_1+C^{\top}\cP_1 C-Q)X+
\cP_1(\wt A \dbE_tX+Bu+\wt B\dbE_tu+b)+
\cP_2 \big[ \sB\dbE_tu+\dbE_tb\big]\\
\ns\ds\qq \q-\big[\sA^{\top}\cP_2+P_1\wt\cA+\wt A^{\top}\cP_1+\wt C^{\top}\cP_1\sC +C^{\top}\cP_1\wt C-\wt Q\big] \dbE_t X\\
\ns\ds\qq\q -\dbE_t \Big[\sA^{\top} \cP_{3} +\wt A^{\top}\cP_{4}
+\wt C^{\top}\cL_4
+\cP_1\wt Bu +\cP_2\big[\sB u+b\big]+\wt C^{\top}\cP_1\big[\sD u+ \si\big]\\
\ns\ds\qq\qq +C^{\top}\cP_1\wt D u\Big]
-\Big\{A^{\top}\cP_{4}+C^{\top}\cL_{4}+\cP_1(B u +b)+C^{\top}\cP_1(D u+\si)\Big\}
\Big\}ds+\sZ dW(s)\\
\ns\ds \qq = \Big\{-\wt A^{\top}\dbE_t \sY-\wt C^{\top}\dbE_t \sZ-
A^{\top} \sY-C^{\top}\sZ+QX+\wt Q\dbE_t X\Big\}ds+\sZ dW(s).
\ea$$
Then for any $t\in[0,T]$, $(\sY,\sZ)\in L^2_{\dbF}(\Omega;C([t,T];\dbR^n))\times L^2_{\dbF}(0,T;\dbR^n)$ satisfies the first backward equation in (\ref{MF-BSDEs-*-1}). By the uniqueness of BSDEs, we see the conclusion.
\end{proof}
\begin{remark}
Given $(Y,Z)\in L^2_{\dbF}(\Omega;C([t,T];\dbR^n))\times L^2_{\dbF}(t,T;\dbR^n)$ satisfying (\ref{MF-BSDEs-*-1}), one can not usually conclude that $( Y_d (\cd), Z_d(\cd))$ is meaningful, where  $ (Y_d(s), Z_d(s))\equiv(Y(s,s),Z(s,s))$, $s\in[0,T]$.
However, Lemma \ref{representation-for-Y-Z-*} indicates that there is a version $(\sY,\sZ)$ such that $\sY_d(\cd)\in L^2_{\dbF}(\Omega;C([0,T];\dbR^n))$ and $\sZ(\cd)\in L^2_{\dbF}(0,T;\dbR^n)$. To our best, the conclusion of Lemma \ref{representation-for-Y-Z-*} is new in the literature.
\end{remark}
To obtain the main result, we need one more result. For $t\in[0,T)$, $x_0\in\dbR^n$, and $u(\cd)\in L^2_{\dbF}(0,T;\dbR^m)$, we consider the following SDE
\bel{The-equation-for-cX}\left\{\2n\ba{ll}
\ns\ds
d\cX(s)=\Big[\sA(s)\cX(s)+\sB (s) u(s) +b(s)\Big]ds+\Big[\sC(s)\cX(s)+\sD (s) u(s) +\si(s)\Big]dW(s),\q
s\in[0,T],\\
\ns\ds \cX(0)=x,\ea\right.\ee
the solvability of which is easy to check under (H1). Unlike the state process $X(\cd)$, $\cX(\cd)$ does not rely on $t$. It is easy to see that
$$\dbP\big\{\omega\in\Omega; \cX(s,\omega)=X(s,\omega),\ \forall s\in[0,t]\big\}=1,$$
where $X(\cd)$ satisfies (\ref{standard-state-equation}).
Given $\cX(\cd)$ satisfying (\ref{The-equation-for-cX}), we denote by
\bel{Definition-of-cM-cN}\ba{ll}
\ns\ds \cM(s,t):=\cP_1(s)\cX(s)+\cP_2(s)\dbE_t\cX(s)+\dbE_t\cP_3(s)+\cP_4(s),\  \ s\in[t,T],\\
\ns\ds \cN(s):=\cP_1(s)(\sC(s) \cX(s)+\sD(s) u(s)+\si(s))+\cL_4(s),\ \ s\in[0,T].
\ea\ee
\bl\label{Lemma-equality}
If $ u(\cd)\in L^2_{\dbF}(0,T;\dbR^m)$. Then for any $t\in[0,T)$,
\bel{The-result-of-lemma-equality}\ba{ll}
\ns\ds \lim_{\e\rightarrow0}\frac 1 \e\dbE_t\int_t^{t+\e}\Big[\sB(s)^{\top}\big[Y_0(s)\g_1 X(t)+\sY(s,t)\big]
+\sD(s)^{\top} \sZ(s,t)\Big]ds\\
\ns\ds =\lim_{\e\rightarrow0}\frac 1 \e\dbE_t\int_t^{t+\e}\Big[\sB(s)^{\top}\big[Y_0(s)\g_1 \cX(s)+\cM(s,s)\big]
+\sD(s)^{\top} \cN(s)\Big]ds.
\ea\ee
%
%
\el

\begin{proof}
First we observe that
\bel{Equality-for-s-B-s-D-0}\ba{ll}
\ns\ds \sB^{\top}  \sY+\sD^{\top}  \sZ \\
\ns\ds=\sB^{\top}\Big[\cP_1 X+\cP_2\dbE_t X+\dbE_t \cP_3+\cP_4\Big] +\sD^{\top}\Big[\cP_1(C X +\wt C\dbE_t X+D  u+\wt D\dbE_t u +\si)+\cL_4\Big]\\
\ns\ds=F_1+ (\sB^{\top}\cP_2+\sD^{\top}\cP_1\wt C) \dbE_t X+\sB^{\top} \dbE_t\cP_3+\sD^{\top}\cP_1\wt D\dbE_t u,
\ea\ee
where
$$
\ba{ll}
\ns\ds  F_1:=\sD^{\top}\big[\cP_1[C  X+D  u+\si]+\cL_4\big] + \sB^{\top}\big[\cP_1 X+\cP_4\big].
\ea
$$
Therefore,
\bel{Integrality-for-s-B-sD-0}\ba{ll}
\ns\ds
\dbE_t\int_t^{t+\e}\big[ \sB^{\top} \sY+ \sD^{\top} \sZ\big]ds =\dbE_t\int_t^{t+\e}F_1 ds+
\dbE_t\int_t^{t+\e}\big( \sB^{\top}\cP_2+\sD^{\top}\cP_1\wt C) Xds\\
\ns\ds\qq\qq\qq\qq\qq\qq\q
+\dbE_t\int_t^{t+\e} \sB^{\top}\cP_3ds+\dbE_t\int_{t}^{t+\e}\sD^{\top}\cP_1\wt D uds.
\ea\ee
Moreover, for any $t\in[0,T)$,
\bel{Inequality-of-Y-0-0}\ba{ll}
\ns\ds \lim_{\e\rightarrow 0} \Big[\frac 1 \e \dbE_t\int_t^{t+\e} \sB^{\top}Y_0\gamma_1\big[ X(s)-X(t)\big]ds\Big]\leq K\lim_{\e\rightarrow 0} \dbE_t\sup_{s\in[t,t+\e]}| X(s)- X(t)|\rightarrow0,\ \ \e\rightarrow0.
\ea\ee
As a result,
\bel{Limit-dbH-two-0}\ba{ll}
\ns\ds  \lim_{\e\rightarrow0}\frac 1 \e\dbE_t\int_t^{t+\e}\Big[\sB(s)^{\top}\big[Y_0(s)\g_1 X(t)+\sY(s,t)\big]
+\sD(s)^{\top} \sZ(s,t)\Big]ds=\lim_{\e\rightarrow0}\frac 1 \e \dbE_t\int_t^{t+\e}\dbG_0(s, X(s))ds,
\ea\ee
where
$$\left\{\ba{ll}
\ns\ds\dbG_0(s, x):=\Big[\sD^{\top}\big[\cP_1[C  x+D u+\si]
+\cL_4\big]+\sB^{\top}\big[\cP_1 x+\cP_4\big]\\
\ns\ds\qq \qq\qq +(\sB^{\top}\cP_2+\sD^{\top}\cP_1\wt C) x+\sB^{\top}
\cP_3+\sD^{\top}\cP_1\wt D u+\sB^{\top}Y_0\gamma_1 x\Big]\\
\ns\ds\qq\qq\ = \sD^{\top}\cP_1\sD u+\cK_1 x+\sD^{\top}\cP_1\si+\sD^{\top}\cL_4
+\sB^{\top}(\cP_4+\cP_3),\\
\ns\ds \cK_1:=\sD^{\top}\cP_1\sC+\sB^{\top}\big(\cP_1+ \cP_2+ Y_0\gamma_1\big).
\ea\right.
$$
On the other hand, by (\ref{Definition-of-cM-cN}), for any $t\in[0,T)$, we have
\bel{Limit-cM-cN-two-0}\ba{ll}
\ns\ds  \lim_{\e\rightarrow0}\frac 1 \e\dbE_t\int_t^{t+\e}\Big[\sB(s)^{\top}\big[Y_0(s)\g_1 \cX(s)+\cM(s,s)\big]
+\sD(s)^{\top} \cN(s)\Big]ds=\lim_{\e\rightarrow0}\frac 1 \e \dbE_t\int_t^{t+\e}\dbG_0(s, \cX(s))ds.
\ea\ee
Recall the equation of $\cX(\cd)$ and $X(\cd)$, for any $t\in[0,T)$ we know that
$$\ba{ll}
\ns\ds \big[X(s)-\cX(s)\big]=\int_t^{s}\Big[\sA(X-\cX)+\wt A(\dbE_t X-X)+\wt B(\dbE_t u-u)\Big]dr\\
\ns\ds\qq\qq \qq\qq +\int_t^{s}\Big[\sC(X-\cX)+\wt C(\dbE_t X-X)+\wt D(\dbE_t u-u)\Big]dW(r).
\ea
$$
As a result, we see that $\dbE_t\big[X(s)-\cX(s)\big]=0$ with $s\in[t,t+\e)$. Hence
$$\ba{ll}
\ns\ds \lim_{\e\rightarrow0}\frac 1 \e \dbE_t\int_t^{t+\e}\dbG_0(s, X(s))ds=\lim_{\e\rightarrow0}\frac 1 \e \dbE_t\int_t^{t+\e}\dbG_0(s, \cX(s))ds.
\ea
$$
Our conclusion is followed by (\ref{Limit-dbH-two-0}) and (\ref{Limit-cM-cN-two-0}).
\end{proof}
\br
Thanks to the Markovian framework and the appearance of conditional expectation operator $\dbE_{\cd}$, we obtain Lemma \ref{Lemma-equality} and transform the investigation of $(X,\sY,\sZ)$ into that of $(\cX,\cM,\cN)$.
\er
\br
If $\wt A=\wt B=\wt C=\wt D=0,$ then $(X,Y,Z)\equiv(\cX,\cM,\cN)$, and Lemma \ref{Lemma-equality} becomes trivial. In other words, above conclusion is necessary only when the mean-field terms in state equation appears.
\er

At this moment, we are ready to give the characterization of open-loop equivalent controls.
\bt\label{Equivalence-open}
Suppose (H1) holds, $Y_0(\cd )$ is in (\ref{Definition-of-Y-0}), $\h P_1(\cd)$ satisfies (\ref{Equation-for-hat-P-1}). Then $\bar u(\cd)$ is an open-loop equilibrium control associated with initial state $x\in\dbR^n$ if and only if
\bel{necessity-condition-on-R-D-c-K}\ba{ll}
\ns\ds \sR(s)-\sD(s)^{\top}\h P_1(s)\sD(s)\geq0,\qq s\in[0,T], \ \ a.e. \ \
\ea\ee
and given $(\bar \cM,\bar\cN)$ in (\ref{Definition-of-cM-cN}) associated with $\bar u(\cd) $,
\bel{One-necessary-general-condition-0}
\ba{ll}
\ns\ds \Big[\sR(s) \bar u(s)- \sB(s)^{\top}\big[Y_0(s)\g_1\bar\cX(s)+\bar \cM(s,s)\big]
-\sD(s)^{\top} \bar \cN(s)\Big]=0, \ \ s\in[0,T]. \ \ a.e.
\ea
\ee
\et
\begin{proof}
From Lemma \ref{first-preparing-lemma}, Lemma \ref{Second-preparing-lemma}, Lemma \ref{representation-for-Y-Z-*}, $\bar u(\cd)$ is an equilibrium control if and only if
\bel{Inequality-condition-for-J-0}\ba{ll}
\ns\ds 0\leq \lim_{\e\rightarrow0}\frac{J(t,\bar X(t);u^{v,\e}(\cd))-J\big(t,\bar X(t);\bar u(\cd)\big)}{\e}
=\lan \sP_0(t) v,v\ran+\lan \sH_0(t),v\ran,
\ea\ee
with any $v\in L^2_{\cF_t}(\Omega;\dbR^m)$. Here $\sP_0(\cd),\ \sH_0(\cd)$ are defined as,
\bel{Some-definitions-of-sP-sH}\left\{\ba{ll}
\ns\ds \sP_0(t):=\lim_{\e\rightarrow0}\frac{1}{2\e}  \int_t^{t+\e}
\big[\sR(s)-\sD(s) ^{\top}\h P_1(s)\sD(s) \big]ds, \ \ \sH_0(t):=\lim_{\e\rightarrow0}\frac 1 \e
\dbE_t\int_t^{t+\e}\dbH_0(s,t)ds,\\
\ns\ds \dbH_0(s,t):=\sR(s) \bar u(s)-\sB(s)^{\top}\big[Y_0(s)\g_1 \bar X(t)+\bar\sY(s,t)\big]
-\sD(s)^{\top} \bar\sZ(s,t),
\ea\right.\ee
and $(\bar\sY,\bar\sZ)$ is in  (\ref{Arbitrary-s-Y-Z-general}) associated with $\bar u(\cd)$.
Furthermore, given $t\in[0,T]$, condition (\ref{Inequality-condition-for-J-0}) holds if and only if both the first-order necessary condition $\sH_0(t)=0$ and second-order condition $\sP_0(t)\geq0$ are true. In other words,
$$
\ba{ll}
\ns\ds \lim_{\e\rightarrow0}\frac{1}{2\e}  \int_t^{t+\e}
\big[\sR(s)-D(s)^{\top}\h P_1(s)D(s)\big]ds\geq 0,\ \ \lim_{\e\rightarrow0}\frac 1 \e
\dbE_t\int_t^{t+\e}\dbH_0(s,t)ds=0.
\ea$$

$\Longrightarrow$
Since both $\sR(\cd)$ and $\h P_1(\cd)$ are bounded and deterministic, we thus know that
$$
0\leq \lim_{\e\rightarrow0}\frac{1}{2\e}  \int_t^{t+\e}
\big[\sR(s)-D(s)^{\top}\h P_1(s)D(s)\big]ds=\sR(t)-D(t)^{\top}\h P_1(t)D(t),\ \ t\in[0,T].\ \ a.e.
$$
If $\sH_0(t)=0$, by Lemma \ref{Lemma-equality}, for any $t\in[0,T)$, one has
\bel{}\ba{ll}
\ns\ds \lim_{\e\rightarrow0}\frac 1 \e\dbE_t\int_t^{t+\e}\Big[\sR(s)\bar u(s)-\sB(s)^{\top}\big[Y_0(s)\g_1 \bar\cX(s)+\bar\cM(s,s)\big]
-\sD(s)^{\top} \bar\cN(s)\Big]ds=0.
\ea\ee
By means of Lemma 3.4 in \cite{Hu-Jin-Zhou-2017}, we see that (\ref{One-necessary-general-condition-0}).

$\Longleftarrow$ If (\ref{One-necessary-general-condition-0}) is true with some $\bar u(\cd)$, we immediately obtain $\sH_0(t)=0$ by means of Lemma \ref{Lemma-equality}. Moreover, (\ref{necessity-condition-on-R-D-c-K}) leads to $\sP_0(t)\geq 0$ with any $t\in[0,T)$. Hence we see that $\bar u(\cd)$ is an equilibrium control.
\end{proof}
In our context, above (\ref{One-necessary-general-condition-0}) and (\ref{necessity-condition-on-R-D-c-K}) are named as the $first$-$order$, $second$-$order$ $equilibrium$ $condition$, respectively.
\br
To our best, the introduced $second$-$order$ $equilibrium$ $condition$ (\ref{necessity-condition-on-R-D-c-K}) has not been discussed in other papers on time inconsistent optimal control problems. Moreover, it is the same as second-order necessary optimality condition of mean-field SDEs (\cite{Buckdahn-Djehiche-Li-2011}), which takes us by surprise.
\er
\br
If the mean-field terms in the state equation disappears, Theorem \ref{Equivalence-open} reduces to the counterparts in \cite{Wang-2017-JDE}. If moreover $\wt R,\ \wt Q=0$, $Q,\ R,\ G$ are definite, then (\ref{necessity-condition-on-R-D-c-K}) is obvious to see, and Theorem \ref{Equivalence-open} becomes consistent with Theorem 3.5 in \cite{Hu-Jin-Zhou-2017}.
\er

\subsection{The case of equilibrium controls with closed-loop representations}

In this part, we characterize the closed-loop representations of open-loop equilibrium controls in sense of Definition \ref{Definition-2}.

For any $(\Th,\f)\in L^2(0,T;\dbR^{m\times n})\times L^2_{\dbF}(0,T;\dbR^m)$, we define
\bel{Simplified-notations-2}\ba{ll}
\ns\ds \cA:=A+B\Th,\  \ \wt\cA:=\wt A+\wt B\Th,\  \ \cC:=C+D\Th,\ \ \wt\cC:=\wt C+\wt D\Th.
\ea\ee
We first give the equation satisfied by $(\cM,\cN)$.
Recall (\ref{Equations-for-P-i-in-uniqueness}), using It\^{o}'s formula, we see that,
\bel{Backward-system-for-cY-cZ}\left\{\!\!\ba{ll}
\ns\ds  d\cM
=-\Big[A^{\top}\cM+C^{\top}\cN+\wt A^{\top}\dbE_t\cM+\wt C^{\top}\dbE_t\cN+Q_1\cX+Q_2\dbE_t\cX+Q_3 u+Q_4\dbE_t u \Big] ds +\cN dW(s),\\
\ns\ds \cM(T,t)=-G\cX(T)-\wt G\dbE_t \cX(T)-\gamma_2,
\ea\right.\ee
where $u(\cd)\in L^2_{\dbF}(0,T;\dbR^m)$, $Q_i$ are bounded deterministic functions defined as,
\bel{Definitions-of-Q-1--4}\ba{ll}
\ns\ds
Q_1:=-(Q+C^{\top}\cP_1\wt C+\cP_1\wt A),\ \ Q_2:=-(\wt Q-C^{\top}\cP_1\wt C-\cP_1\wt A), \\
\ns\ds Q_3:=-(C^{\top}\cP_1\wt D+\cP_1\wt B),\ \ Q_4:=(C^{\top}\cP_1\wt D+\cP_1\wt B).
\ea
\ee
Next we decouple the forward-backward system of $(\cX,\cM,\cN)$ if $u(\cd)=\Th(\cd)\cX(\cd)+\f(\cd)$ with proper $(\Th,\f)$. Using similar ideas as the ones from (\ref{Representation-for-Y-*}) to (\ref{Equations-for-P-i-in-uniqueness}), for  $s\in[0,T]$, we consider the following systems of equations
\bel{open-loop-closed-loop-representation}\left\{
\ba{ll}
\ns\ds dP_1=- \Big[P_1\sA+P_1\sB\Th+A^{\top}P_1+C^{\top}P_1(\sC+\sD\Th)-Q_1-Q_3\Th\Big]ds,\ \ \\
\ns\ds dP_2=-\Big\{P_2\sA+P_2\sB\Th+\sA^{\top}P_2-Q_2-Q_4\Th+\wt A^{\top}P_1 +\wt C ^{\top} P_1(\sC+\sD\Th)\Big\}ds,\\
\ns\ds dP_{3}=-\Big[\sA^{\top}P_3+P_2(\sB \f+b)-Q_4\f+\wt A^{\top} P_4+\wt C^{\top}P_1\sD\f+\wt C^{\top}P_1\si+\wt C^{\top}\L_4\Big]ds+\L_{3}dW(s),\\
dP_4=-\Big\{A^{\top}P_4+C^{\top}\L_4+P_1(\sB \f+b)+C^{\top}P_1\sD\f+C^{\top} P_1\si-Q_3\f\Big\}ds+\L_4dW(s),
\\
\ns\ds P_1(T)=-G,\ P_2(T)=-\wt G,\ P_3(T)=0,\ P_4(T)=-\gamma_2,
\ea\right.\ee
where $Q_i(\cd)$ are defined in (\ref{Definitions-of-Q-1--4}). By the standard theory of BSDEs, it is easy to see
\bel{regularity-of-P-i}
\ba{ll}
\ns\ds P_1(\cd),\ P_2(\cd)\in C([0,T];\dbR^{n\times n}),\ (P_{3 },\L_{3}),(P_{4},\L_{4})\in  L^2_{\dbF}(\Omega;C([0,T];\dbR^{n}))\times L^2_{\dbF}(0,T;\dbR^{n}).
\ea
\ee
For $s\in[0,T]$ and $t\in[0,s]$, we define
\bel{Arbitrary-s-Y-Z-general-0}\ba{ll}
\ns\ds  \sM(s,t):=P_1(s)\cX(s)+  P_2(s)\dbE_t\cX(s)+\dbE_t P_{3}(s)+ P_{4}(s),\\
\ns\ds \sN(s):= P_1(s)\big(\sC(s)+\sD(s)\Th(s)\big) \cX(s) +P_1(s)\sD(s) \f (s)+P_1(s)\si(s)+\L_{4}(s).
\ea
\ee
For any $t\in[0,T)$, it is easy to see that $(\sM,\sN)\in L^2_{\dbF}(\Omega;C([t,T];\dbR^n))\times L^2_{\dbF}(0,T;\dbR^n)$.
\bl\label{representation-for-Y-Z-*-0-new}
\rm Suppose $(\cX,\cM,\cN)$ is the unique pair of solution for forward-backward system (\ref{The-equation-for-cX}), (\ref{Backward-system-for-cY-cZ}), and $u(\cd)$ admits a linear form of $u(\cd)= \Th(\cd) \cX(\cd)+\f(\cd)$, where $(\Th(\cd),\f(\cd))\in L^2(0,T;\dbR^{m\times n})\times L^2_{\dbF}(0,T;\dbR^m)$. Then for any $t\in[0,T]$,
\bel{representation-for-any-y-z}\ba{ll}
\ns\ds   \dbP\Big\{\omega\in\Omega;\ \cM(s,t)=\sM(s,t),\ \ \forall s\in[t,T] \Big\}=1,\\
\ns\ds \dbP\Big\{\omega\in\Omega;\  \cN(s) = \sN(s)\Big\}=1,\ \ s\in[0,T]. \ a.e.
\ea
\ee
\el
\begin{proof}
Given (\ref{Arbitrary-s-Y-Z-general-0}), it is easy to see that
$$\left\{
\ba{ll}
\ns\ds \dbE_t\sM =(P_1+P_2)\dbE_t\cX+  \dbE_t[P_3+P_4],\\
\ns\ds \dbE_t\sN =P_1\big(\sC+\sD\Th\big) \dbE_t\cX +P_1\sD \dbE_t\f +P_1\dbE_t\si+\dbE_t\L_{4}.
\ea\right.$$
For notational simplicity, we define the generators in (\ref{open-loop-closed-loop-representation}) as $\Pi_i$, $i=1,2,3,4$. In other words,
$$
\ba{ll}
\ns\ds  d P_i(s)=\Pi_{i}(s)ds,\ \ i=1,2, \ \ d P_j(s)=\Pi_j(s)ds+\L_j(s)dW(s),\ \ j=3,4.
\ea
$$
Using It\^{o}'s formula, we derive that
$$\left\{
\ba{ll}
\ns\ds d\big[  P_1 \cX\big]=\Big[\Pi_{1} \cX+ P_1(\sA \cX  +\sB\Th \cX+\sB\f +b)\Big]ds
 + P_1\big(\sC \cX+
\sD \Th \cX+\sD\f  +\si\big)dW(s),\\
\ns\ds d\big[P_2\dbE_t \cX\big]=\Big\{\Pi_2 \dbE_t \cX+P_2 \big[(\sA+\sB\Th)\dbE_t\cX+\sB\dbE_t\f+\dbE_tb\big]\Big\}ds.
\ea\right.
$$
As a result, we have
$$
\ba{ll}
\ns\ds d\sM=\Big\{\Pi_{1} \cX+ P_1(\sA+\sB\Th)\cX+P_1\sB \f+P_1b
+\Pi_2 \dbE_t \cX
\\
\ns\ds\qq\ +P_2 \big[(\sA+\sB\Th)\dbE_t\cX+\sB\dbE_t\f+\dbE_tb\big]
+\dbE_t\Pi_3+\Pi_4 \Big\}ds\\
\ns\ds\qq\ +\Big[ P_1\big(\sC \cX +\sD \Th \cX+\sD\f +\si\big)+\L_{4}\Big]dW(s)\\
\ns\ds\qq =\Big\{(\Pi_1+P_1\sA+P_1\sB\Th)\cX+(\Pi_2+P_2\sA+P_2\sB\Th)\dbE_t\cX
+\dbE_t\Big[P_2(\sB \f+b)+\Pi_3\Big]\\
\ns\ds\qq\q +P_1(\sB \f+b)+\Pi_4
\Big\}ds+\Big[ P_1\big(\sC \cX +\sD\Th \cX+\sD \f +\si\big)+\L_{4}\Big]dW(s).
\ea
$$
On the other hand, by the previous representations,
$$
\ba{ll}
\ns\ds -
A^{\top}\sM-C^{\top}\sN -
\wt A^{\top}\dbE_t\sM-\wt C^{\top}\dbE_t\sN+Q_1\cX+ Q_2\dbE_t \cX+Q_3 (\Th \cX+\f)+Q_4\dbE_t (\Th \cX+\f)\\
\ns\ds= -A^{\top}\Big\{P_1\cX+  P_2\dbE_t\cX +\dbE_t P_{3}+ P_{4} \Big\}
-C^{\top} \Big\{P_1\big(\sC+\sD\Th\big) \cX +P_1\sD \f +P_1\si+\L_{4}\Big\}\\
\ns\ds\q-\wt A^{\top}\Big[(P_1+P_2)\dbE_t \cX +\dbE_t \big[P_{3}+P_{4}\big]\Big]
-\wt C^{\top}\Big[P_1\big(\sC+\sD\Th\big)\dbE_t \cX +P_1\sD \dbE_t\f +P_1\dbE_t\si+\dbE_t\L_{4} \Big]\\
\ns\ds\qq +(Q_1+Q_3\Th)\cX+ (Q_2+Q_4\Th)\dbE_t\cX +Q_3\f+Q_4\dbE_t\f\\
\ns\ds =-\Big[A^{\top}P_1+C^{\top}P_1(\sC+\sD\Th)-Q_1-Q_3\Th\Big]\cX\\
\ns\ds
-\Big[A^{\top}P_2-Q_2-Q_4\Th+\wt A^{\top}(P_1+P_2)+\wt C ^{\top} P_1(\sC+\sD\Th)\Big]\dbE_t\cX\\
\ns\ds
-\dbE_t\Big[A^{\top}P_3-Q_4\f+\wt A^{\top}(P_3+P_4)+\wt C^{\top}P_1\sD\f+\wt C^{\top}P_1\si+\wt C^{\top}\L_4\Big]\\
\ns\ds
-A^{\top}P_4-C^{\top}P_1\sD\f-C^{\top}(P_1\si+\L_4)+Q_3\f.
\ea
$$
At this moment, by the choice of $\Pi_i(\cd)$, as well as (\ref{Arbitrary-s-Y-Z-general-0}), we conclude that $(\sM,\sN)$ satisfies the backward equation in (\ref{Backward-system-for-cY-cZ}) with $t\in[0,T]$. By the uniqueness of BSDEs, we see the conclusion.
\end{proof}
Now we give the main result of this part.
\begin{theorem}\label{characterization}
Suppose (H1) holds true, $\h P_1(\cd)$ satisfies (\ref{Equation-for-hat-P-1}). For any $x\in\dbR^n$, the linear quadratic problem admits an open-loop equilibrium control $u^*(\cd)$ in the sense of Definition \ref{Definition-2} associated with $(\Th^*,\f^*)\in L^2(0,T;\dbR^{m\times n})\times L^2_{\dbF}(0,T;\dbR^m)$ if and only if (\ref{necessity-condition-on-R-D-c-K}) is true
and there exist $ P_1^*(\cd),P_2^*(\cd),$ $(P_3^*(\cd),\L^*_3(\cd))$, $(P_4^*(\cd),\L^*_4(\cd))$ satisfying BSDEs of (\ref{open-loop-closed-loop-representation}) with
\bel{representations-for-Th-f}\ba{ll}
\ns\ds \Th^*=\big[\sR-\sD^{\top}P_1^* \sD\big]^{\dagger}\[\sB^{\top}(P_1^*+P_2^*+Y_0 \g_1)+\sD^{\top}P_1^* \sC\]\\
\ns\ds\qq +\Big\{I-\big[\sR-\sD^{\top}P_1^*\sD\big]^{\dagger}\big[\sR-\sD^{\top}P_1^*\sD\big]\Big\}\th,\\
\ns\ds \f^*= \Big[\big[\sR-\sD^{\top}P_1^*\sD\big]^{\dagger}\[\sB^{\top}[P_4^*+P_3^*]
+\sD^{\top}[P_1^*\si +\L_4^*]\]\\
\ns\ds \qq +\Big\{I-\big[\sR-\sD^{\top}P_1^*\sD\big]^{\dagger}\big[\sR-\sD^{\top}P_1^*\sD\big]\Big\}\f,
\ea\ee
such that the following hold,
\bel{Some-furthermore-condition-on}\ba{ll}
\ns\ds \cR\Big(\sB^{\top}(P_1^*+P_2^*+Y_0 \g_1)+\sD^{\top}P_1^* \sC\Big)\subset \cR\Big(\sR-\sD^{\top}P_1^* \sD\Big),\ \ a.e. \\
\ns\ds \Big[\sB^{\top}[P_4^*+P_3^*]
+\sD^{\top}[P_1^*\si +\L_4^*]\Big]\in \cR\Big(\sR-\sD^{\top}P_1^* \sD\Big), \ a.e. \ a.s. \\
\ns\ds \big[\sR-\sD^{\top}P_1^* \sD\big]^{\dagger}\[\sB^{\top}(P_1^*+P_2^*+Y_0 \g_1)+\sD^{\top}P_1^* \sC\]\in L^2(0,T;\dbR^{m\times n}),\\
\ns\ds \Big[\big[\sR-\sD^{\top}P_1^*\sD\big]^{\dagger}\[\sB^{\top}[P_4^*+P_3^*]
+\sD^{\top}[P_1^*\si +\L_4^*]\] \in L^2_{\dbF}(0,T;\dbR^m).
\ea\ee
Here $\th(\cd)\in L^2(0,T;\dbR^{n\times n})$, $\f(\cd)\in L^2_{\dbF}(0,T;\dbR^m) $, $A^{\dagger}$ represents the pseudo-inverse of matrix $A$.
\end{theorem}
\begin{proof}
$\Longrightarrow$ Given $(\Th^*,\f^*)$, from (\ref{regularity-of-P-i}), we see the existence of $(P_1^*,P_2^*), \ (P^*_3(\cd),\L^*_3(\cd)),\ (P^*_4(\cd),\L^*_4(\cd))$ satisfying (\ref{open-loop-closed-loop-representation}).
By Theorem \ref{Equivalence-open}, we only need to prove (\ref{representations-for-Th-f}) , (\ref{Some-furthermore-condition-on}). From (\ref{One-necessary-general-condition-0}), (\ref{Arbitrary-s-Y-Z-general-0}) and Lemma \ref{representation-for-Y-Z-*-0-new},
\bel{To-obtain-the-representation-of-two-1}
\ba{ll}
\ns\ds  0=\Big[\sR(s) u^*(s)- \sB(s)^{\top}\big[Y_0(s)\g_1 \sX^*(s)+  \cM^*(s,s)\big]
-\sD(s)^{\top}   \cN^*(s)\Big]\\
\q=\Big[\sR(s) u^*(s)- \sB(s)^{\top}\big[Y_0(s)\g_1 \sX^*(s)+  \sM^*(s,s)\big]
-\sD(s)^{\top}   \sN^*(s)\Big]\\
\ns\ds \q=\Big[\big[\sR-\sD^{\top}P_1^*\sD\big]\Th^*-
\sB^{\top}\big[Y_0\gamma_1+P_1^*+P_2^*\big]-\sD P_1^*\sC\Big]\sX^*\\
\ns\ds\qq+\big[\sR-\sD^{\top}P_1^*\sD\big]\f^*-\sD^{\top}\big[P_1^*\si+\L_4^*\big]-\sB^{\top}
\big[P_3^*+P_4^*\big],
\ea
\ee
where $\sX^*(\cd)$ satisfies (\ref{Closed-loop-representation-state-equation}). Notice (\ref{To-obtain-the-representation-of-two-1}) holds true for any $x\in\dbR^n$. We denote by $\sX^*_0(\cd)$ be the solution of (\ref{Closed-loop-representation-state-equation}) associated with initial state $x=0$. Then one has
$$\ba{ll}
\ns\ds \Big[\big[\sR-\sD^{\top}P_1^*\sD\big]\Th^*-
\sB^{\top}\big[Y_0\gamma_1+P_1^*+P_2^*\big]-\sD P_1^*\sC\Big](\sX^*-\sX^*_0)=0.
\ea
$$
At this moment, given unit matrix $I\in\dbR^{n\times n}$, we consider the following equation
\bel{Closed-loop-state-equation-standard}\left\{\2n\ba{ll}
\ns\ds
d\sX^*_1(s)= \big[\sA(s)+\sB(s)\Th^*(s)\big]\sX^*_1(s) ds + \big[\sC(s)+\sD(s)\Th^*(s)\big]\sX^*_1(s) dW(s),\q
s\in[0,T],\\
\ns\ds \sX^*_1(0)=I.\ea\right.\ee
It is easy to see the solvability of $\sX^*_1(\cd)$. Moreover, $\big[\sX^*_1\big]^{-1}(\cd)$ exists. By the uniqueness of SDEs, for any $x\in\dbR^n$ we know that
$$
\ba{ll}
\ns\ds \dbP\big\{\omega;\ \sX^*(s,\omega)-\sX^*_0(s,\omega)=\sX^*_1(s,\omega)x,\ \forall s\in[0,T] \big\}=1.
\ea
$$
By the existence of $\big[\sX^*_1\big]^{-1}(\cd)$, we know that
\bel{One-equality-to-obtain-Th}\ba{ll}
\ns\ds \Big[\big[\sR-\sD^{\top}P_1^*\sD\big]\Th^*-
\sB^{\top}\big[Y_0\gamma_1+P_1^*+P_2^*\big]-\sD P_1^*\sC\Big]=0,
\ea
\ee
which then leads to
\bel{One-equality-to-obtain-varphi}
\ba{ll}
\ns\ds \big[\sR-\sD^{\top}P_1^*\sD\big]\f^*-\sD^{\top}\big[P_1^*\si+\L_4^*\big]-\sB^{\top}
\big[P_3^*+P_4^*\big]=0.
\ea
\ee
Then we can see the conclusions in (\ref{representations-for-Th-f}), (\ref{Some-furthermore-condition-on}).

$\Longleftarrow$ For any $x\in\dbR^n$, we define
$u^*(\cd):=\Th^*(\cd) \sX^*(\cd)+\f^*(\cd)$, where $\sX^*(\cd)$ satisfies (\ref{Closed-loop-representation-state-equation}).
By the choice of $\Th^*(\cd),\f^*(\cd)$, we have above (\ref{To-obtain-the-representation-of-two-1}). The conclusion is easy to see via Theorem \ref{Equivalence-open}.
\end{proof}
\br
As to systems of equations (\ref{open-loop-closed-loop-representation}), we look at the special case when $\wt A=\wt B=\wt C=\wt D=0$,
\bel{Special-case-I}\left\{
\ba{ll}
\ns\ds dP_1=- \Big[P_1 A+A^{\top}P_1+C^{\top}P_1C-Q+(P_1 B+C^{\top}P_1D)\Th\Big]ds,\ \ \\
\ns\ds dP_2=-\Big\{P_2 A+A^{\top}P_2-\wt Q+P_2 B\Th\Big\}ds,\\
\ns\ds dP_{3}=-\Big[ A^{\top}P_3+P_2(B \f+b)\Big]ds+\L_{3}dW(s),\\
dP_4=-\Big\{A^{\top}P_4+C^{\top}\L_4+P_1(B \f+b)+C^{\top}P_1 D\f+C^{\top} P_1\si \Big\}ds+\L_4dW(s),
\\
\ns\ds P_1(T)=-G,\ P_2(T)=-\wt G,\ P_3(T)=0,\ P_4(T)=-\gamma_2,
\ea\right.\ee

$\bullet$ If $b=\si=\gamma_2=0$, $\f=0,$ then $(P_3,P_4)\equiv(0,0)$. The equations of $(P_1^*,P_2^*)$ associated with $\Th^*$ is just the system of Riccati equations in \cite{Yong-2017}.

$\bullet$ If $b=\si\equiv0$, $\wt R=\wt Q=0,$ $m=n=1,$ $\f^*$ is deterministic, then $\L_3=\L_4\equiv 0$, $P_3^*, P_4^*$ satisfy deterministic backward ODEs. This corresponds to the case in \cite{Hu-Jin-Zhou-2012}, \cite{Hu-Jin-Zhou-2017}. Notice that the randomness of $b,\si$ leads to above BSDEs which appears for the first time to our best.
\er
\br
We look at another case of (\ref{open-loop-closed-loop-representation}) when $A=B=C=D=0$:
$$\left\{
\ba{ll}
\ns\ds d\wt P_1=- \Big[\wt P_1\wt A -Q_1+(\wt P_1\wt B-Q_3)\Th\Big]ds,\ \ \\
\ns\ds d\wt P_2=-\Big\{\wt P_2\wt A+\wt C ^{\top}\wt  P_1\wt C+\wt A^{\top}\wt P_2-Q_2+\wt A^{\top}\wt P_1 +(\wt P_2\wt B-Q_4 +\wt C ^{\top} \wt P_1\wt D)\Th\Big\}ds,\\
\ns\ds d\wt P_{3}=-\Big[\wt A^{\top}\wt P_3+\wt C^{\top}\wt P_1\si+\wt P_2b+(\wt P_2\wt B+\wt C^{\top}\wt P_1\sD-Q_4)\f+\wt A^{\top} \wt P_4+\wt C^{\top}\wt \L_4\Big]ds+\wt \L_{3}dW(s),\\
d\wt P_4=-\Big\{ (\wt P_1\wt B-Q_3) \f+\wt P_1b\Big\}ds+\wt \L_4dW(s),
\\
\ns\ds \wt P_1(T)=-G,\ \wt P_2(T)=-\wt G,\ \wt P_3(T)=0,\ \wt P_4(T)=-\gamma_2,
\ea\right.$$
Here we use the term $\wt P_i$ instead.

$\bullet$ The coefficients of above system of equations relies on the mean-field terms. To our best, such system is new in the related literature.

$\bullet$ We point out one more interesting thing. If we add the first two equations together, we see that
\bel{One-more-interesting-thing}\left\{\ba{ll}
\ns\ds d(\wt P_1+\wt P_2)=-\Big[(\wt P_1+\wt P_2)\wt A+\wt A^{\top}(\wt P_1+\wt P_2)+\wt C^{\top}\wt P_1\wt C-Q-\wt Q+(\wt P_1\wt B+\wt P_2\wt B+\wt C^{\top}\wt P_1\wt D)\Th\Big],\\
\ns\ds \wt P_1(T)+\wt P_2(T)=-(\wt G+G).
\ea\right.
\ee
It is a direct calculation that $(P_1+P_2)$ in (\ref{Special-case-I}) satisfies the same equation as (\ref{One-more-interesting-thing}) if we replace $\wt A,\wt B,\wt C,\wt D$ with $A,B,C,D$. One can also obtain similar conclusion for $(P_3+P_4)$ and $(\wt P_3+\wt P_4)$.
\er
\br
We observe that the introduced terms $\wt B$, $\wt D$ in (\ref{standard-state-equation}) could bring essential influence on $(\Th^*,\f^*)$ in (\ref{representations-for-Th-f}). Here are some special cases.

$\bullet$ In terms of (\ref{representations-for-Th-f}), both $\wt D $ and $D$ play important roles in allowing $R$, $\wt R$ to be indefinite. In other words, if $\wt R=R=D=0$, equilibrium control $u^*$ could still have feedback form under proper conditions.

$\bullet$ Suppose $\sR\geq\d>0$ and $B=D=0$. From (\ref{representations-for-Th-f}) we see that the feedback form could still make sense by imposing suitable conditions on $\wt B$, $\wt D$.

$\bullet$ Suppose $\sR\geq\d>0$, and $B\neq0$ or $D\neq0$. According to (\ref{representations-for-Th-f}), the terms $\wt B$, $\wt D$ could deny the existence of feedback form for $u^*$, if for example, $\wt B:=- B$, $\wt D:=- D$.

\er

\section{Uniqueness of open-loop equilibrium controls}

In this section, we study the uniqueness of open-loop equilibrium controls.

\ms

(H2) Suppose equilibrium control $u^*$ has a representation of $u^*=\Th^*\sX^*+\f^*$, and $\sR -\sD^{\top} P_1^* \sD \geq\d>0$.

\ms

Given $P_i^*$, $(\cM,\cN)$ in (\ref{open-loop-closed-loop-representation}), (\ref{Backward-system-for-cY-cZ}) associated with $u^*(\cd)$, $u(\cd)\in L^2_{\dbF}(0,T;\dbR^m)$, respectively, we define
\bel{definition-for-arbitrary-y-z}
\ba{ll}
\ns\ds \cK(s,t):=\cM(s,t)-\big[P_1^*\cX+P_2^*\dbE_t\cX+\dbE_t P_3^*+P_4^* \big](s),\ \ s\in[t,T],\\
\ns\ds \cH(s):=\cN(s)-\big\{P_1^*(\sC X +\sD u +\si) +\L_4\big\}(s),\ \ s\in[0,T].
\ea
\ee
Recall the definition of $(\cM,\cN)$ in (\ref{Definition-of-cM-cN}), it is easy to see that for $ \cK_{d}(s)\equiv\cK(s,s)$ with $s\in[0,T]$,
\bel{arbitrary-diag-s-Y-Z-difference}\ba{ll}
\ns\ds  \cK_{d}(\cd)\in L^2_{\dbF}(\Omega;C([0,T];\dbR^n)),\ \ \cH(\cd)\in L^2_{\dbF}(0,T;\dbR^n).
\ea\ee
Given $(\bar\cK_d,\bar\cH)$ associated with $\bar u$, and $\cP_i$ in (\ref{Equations-for-P-i-in-uniqueness}), we define
\bel{definitions-of-c-G}\left\{\ba{ll}
\ns\ds \cG_1 :=\cG_3\big[
\sB ^{\top}\bar \cK_d +\sD ^{\top}\bar \cH \big], \ \  \cG_2 := \cG_4\big[
\sB ^{\top}\bar \cK_d +\sD ^{\top}\bar \cH \big],\\
\ns\ds \cG_3:=(C^{\top}P_1^*\sD+P_1^*\sB+C^{\top}\cP_1\wt D+\cP_1\wt B) [\sR- \sD^{\top}P_1^* \sD]^{-1},\\
\ns\ds \cG_4:=(\wt C^{\top}P_1^*\sD+P_2^*\sB-C^{\top}\cP_1\wt D-\cP_1\wt B)[\sR-\sD^{\top}P_1^*\sD]^{-1}.
\ea\right.
\ee
\begin{lemma}\label{Lemma-representation}
Suppose (H1)-(H2) are true and $\bar u(\cd)$ is another equilibrium control. Then
\bel{One-pre-representation-for-bar-u}
\ba{ll}
\ns\ds \bar u(s)=\Th^*(s) \bar \cX(s)+\f^*(s)+\big[\sR(s)-\sD(s)^{\top}P_1^*(s)\sD(s)\big]^{-1}
\big[\sB(s)^{\top}\bar \cK_d(s)+\sD(s)^{\top}\bar \cH(s)\big],\ \ s\in[0,T],
\ea
\ee
where $(\bar\cK,\bar\cH)$ is the solution of the following backward equation
\bel{first-step-equation-for-bar-c-y-c-z}
\left\{\ba{ll}
\ns\ds d\bar\cK=
-\Big\{ A^{\top}\bar\cK+C^{\top} \bar\cH +\wt A^{\top}\dbE_t\bar\cK+\wt C^{\top}\dbE_t\bar\cH+\cG_1+\dbE_t\cG_2\Big\}ds+\bar\cH dW(s),\ \ s\in[t,T],\\
\ns\ds \bar \cY(T,t)=0.
\ea\right.
\ee
\end{lemma}
\begin{proof}
By It\^{o}'s formula, we have
$$\left\{
\ba{ll}
\ns\ds d\[P_1^* \cX\]=
\Big\{-A^{\top}\[P_1^*\cX\]-C^{\top}P_1^*(\sC+\sD\Th^*)\cX+Q_1\cX+Q_3\Th^*\cX
-P_1^*\sB\big[\Th^* \cX-u\big]\\
\ns\ds\qq\qq\q +P_1^*b\Big\}ds+ P_1^*(\sC \cX+\sD u+\si) dW(s),\\
\ns\ds d\big[P_2^* \dbE_t \cX\big]=\Big\{-\sA ^{\top} P_2^* \dbE_t \cX
- P_2^* B\dbE_t(\Th^* \cX-u)+P_2^* \dbE_t b+  (Q_2+Q_4\Th^*)\dbE_t \cX\\
\ns\ds\qq\qq \q -\big[\wt A^{\top}P_1^*+\wt C^{\top}P_1^*(\sC+\sD\Th^*)\big]\dbE_t\cX\Big\}ds.
\ea\right.
$$
Recall $P_i^*$ in (\ref{open-loop-closed-loop-representation}), we see that
$$
\ba{ll}
\ns\ds d\[P_1^*\cX+P_2^*\dbE_t \cX +\dbE_tP_3^*+P_4^*\]\\
\ns\ds =\Big\{-A^{\top}\Big[P_1^*\cX+P_2^*\dbE_t\cX+
\dbE_tP_3^*+P_4^*\Big]-\wt A^{\top}\Big[P_1^*\dbE_t\cX+P_2^*\dbE_t\cX+
\dbE_tP_3^*+\dbE_tP_4^*\Big]
\\
\ns\ds\qq-C^{\top}\Big[P_1^*\big[\sC\cX+\sD u+\si\big]+\L_4^*\Big]
-\wt C^{\top}\Big[P_1^*\dbE_t\big[\sC\cX+\sD  u+ \si\big]+ \L_4^*\Big]\\
\ns\ds\qq -\big[C^{\top}P_1^*\sD+P_1^*\sB\big](\Th^*\cX+\f^*-u)-\big[\wt C^{\top}P_1^*\sD+P_2^*\sB\big]\dbE_t(\Th^*\cX+\f^*-u)\\
\ns\ds\qq +Q_1\cX+Q_2\dbE_t\cX+Q_3(\Th^* \cX+\f^*)+Q_4(\Th^*\dbE\cX+\dbE_t\f^*)\Big\}ds\\
\ns\ds\qq +\big[P_1^*(\sC\cX+\sD u+\si)+\L_4^*\big] dW(s).
\\
\ea
$$
Observe that $(\cM,\cN)$ satisfies the backward equation in (\ref{Backward-system-for-cY-cZ}), thus we can obtain the following equation with respect to $(\cK,\cH)$,
\bel{equation-for-arbitrary-y-z}
\ba{ll}
\ns\ds d\cK=\Big\{-A^{\top}\cK-C^{\top} \cH-\wt A^{\top}\dbE_t\cK-\wt C^{\top} \dbE_t\cH + (C^{\top}P_1^*\sD+P_1^*\sB+C^{\top}\cP_1\wt D+\cP_1\wt B)[\Th^*\cX+\f^*-u]\\
\ns\ds\qq\ + \big[\wt C^{\top}P_1^*\sD +P_2^* \sB-C^{\top}\cP_1\wt D-\cP_1\wt B\big] \dbE_t\big[(\Th^*\cX+\f^*-u)\big]\Big\}ds+\cH dW(s).
\ea
\ee
From Theorem \ref{Equivalence-open} and (\ref{definition-for-arbitrary-y-z}), (\ref{One-equality-to-obtain-Th}), (\ref{One-equality-to-obtain-varphi}), if $\bar u(\cd)$ is an equilibrium control, then
$$
\ba{ll}
\ns\ds 0=\big[\sR \bar u- \sB^{\top}\bar \cM-\sD^{\top}\bar \cN- \sB^{\top}Y_0\gamma_1 \bar\cX\big] \\
\ns\ds\q =- \sB ^{\top}\bar\cK_d -\sD ^{\top}\bar \cH -\big[\sB^{\top}(P_1^*+P_2^* +Y_0\gamma_1)+\sD^{\top}P_1^*\sC\big] \bar\cX  \\
\ns\ds\qq -\big[\sB^{\top}(P_3^*+ P_4^*)+\sD^{\top}[P_1^*\si+ \L_4^*]\big] +(\sR-\sD^{\top}P_1^*\sD) \bar u \\
\ns\ds\q =- \sB ^{\top}\bar\cK_d -\sD ^{\top}\bar \cH -(\sR-\sD^{\top}P_1^*\sD) \big[\Th^* \bar\cX +\f^*-\bar u\big],
\ea
$$
from which  we obtain (\ref{One-pre-representation-for-bar-u}). Plugging it into (\ref{equation-for-arbitrary-y-z}) associated with $\bar u$, we obtain (\ref{first-step-equation-for-bar-c-y-c-z}).
\end{proof}
The following result shows more explicit structure of solution for (\ref{first-step-equation-for-bar-c-y-c-z}).
\begin{lemma}\label{Decoupling-mean-field}
For any $t\in[0,T)$, we consider BSDEs of
\bel{Conditional-MF-BSDEs}\left\{
\ba{ll}
\ns\ds dY=-\big[A^{\top}Y+C^{\top}Z+\wt A\dbE_t Y+\wt C\dbE_tZ+Q_1+\dbE_t Q_2\big]ds+Z dW(s),\ \ s\in[t,T],\\
\ns\ds dY_1=-\Big[A^{\top}Y_1+C^{\top}Z_1+Q_1\Big]ds+Z_1dW(s),\ \ s\in[0,T], \\
\ns\ds dY_2=-\Big\{(A^{\top}+\wt A^{\top})Y_2+\wt A^{\top}Y_1+\wt C^{\top}Z_1+Q_2\Big\}ds+Z_2dW(s),\ \ s\in[0,T],\\
\ns\ds Y(T)=G_1+\dbE_t G_2,\ \ Y_1(T)=G_1,\ \ Y_2(T)=G_2,
\ea\right.
\ee
where
$$A, C,\wt A,\wt C\in L^{\infty}(0,T;\dbR^{n\times n}),\ \ Q_1, Q_2\in L^2_{\dbF}(\Omega;L^1(0,T;\dbR^{n\times n})),\ \ G_1, G_2\in L^{\infty}_{\cF_T}(\Omega;\dbR^{n\times n}).$$
Then for any $t\in[0,T)$, the first equation in (\ref{Conditional-MF-BSDEs}) admits a unique pair of solution
$(Y,Z)\in L^2_{\dbF}(\Omega;C([t,T];\dbR^n))\times L^2_{\dbF}(t,T;\dbR^n)$ such that
$$
Y(s,t)=Y_1(s)+\dbE_t Y_2(s),\ \ \forall s\in[t,T], \ \ Z(r,t)=Z_1(r),\ \ r\in[t,T]. \ \ a.e.
$$
\el

\begin{proof}
Under given requirements, it is easy to see the unique solvability of $(Y_i,Z_i)$ satisfying
$$
(Y_i,Z_i)\in L^2_{\dbF}(\Omega;C([0,T];\dbR^n))\times L^2_{\dbF}(0,T;\dbR^n),\ \ i=1,2.
$$
We define $
Y:=Y_1+\dbE_t Y_2,$ $Z:=Z_1.
$
It is easy to verify that $(Y,Z)$ satisfy equation (\ref{Conditional-MF-BSDEs}).
To verify the uniqueness, let us suppose that $(Y',Z')$ is another pair of solution. Let $M:= Y'-Y,$ $N:=Z'-Z$. Then
\bel{MF-BSDE-zero-solution}\left\{
\ba{ll}
\ns\ds dM=-\big[A^{\top}M+C^{\top}N+\wt A^{\top}\dbE_tM+C^{\top}\dbE_t N\big]ds+N dW(s),\ \ s\in[t,T],\\
\ns\ds M(T)=0.
\ea\right.
\ee
As to (\ref{MF-BSDE-zero-solution}), the standard mean-field BSDEs theory shows the unique solvability of $(M,N)\equiv(0,0)$. Then our conclusion is easy to see.
\end{proof}
At this moment, we are ready to give the uniqueness of open-loop equilibrium control.

\bt
Given equilibrium controls $\bar u$ and $u^*$ such that (H1)-(H2) are true. Then $\bar u(\cd)=u^*(\cd)$.
\et

\begin{proof}
According to Lemma \ref{Lemma-representation}, we only need to prove $(\bar \cK_d,\bar \cH)\equiv (0,0)$. To this end,  given $(\bar \cK_d,\bar \cH )$ satisfying (\ref{arbitrary-diag-s-Y-Z-difference}), $\cG_i$ in (\ref{definitions-of-c-G}), for $s\in[0,T]$ let us consider
$$
\left\{\ba{ll}
\ns\ds dY_1(s)=-\Big[A^{\top}Y_1 +C^{\top} Z_1
+ \cG_1\Big]ds+Z_1 dW(s), \\
\ns\ds dY_2(s)=-\big[(A^{\top}+\wt A^{\top})Y_2+\wt A^{\top}Y_1+\wt C^{\top}Z_1+ \cG_2 \big]ds+Z_2(s)dW(s),\\
\ns\ds Y_1(T)=0,\  \ Y_2(T)=0,
\ea\right.
$$
the solvability of which is easy to see.
By Lemma \ref{Decoupling-mean-field}, for any $t\in[0,T]$, the following is true,
$$\ba{ll}
\ns\ds \dbP\big\{\omega\in\Omega;\bar\cK(s,t)=Y_1(s)+\dbE_t Y_2(s),\ \forall s\in[t,T]\big\}=1, \ \
\dbP\big\{\omega\in\Omega; \bar\cH(s)=Z_1(s)\big\}=1,\ \  s\in[t,T]. \ a.e.
\ea
$$
Therefore, one has
$$
\ba{ll}
\ns\ds \dbP\big\{\omega\in\Omega;\bar\cK_d(t)\equiv\bar\cK(t,t)=Y_1(t)+ Y_2(t)\big\}=1, \ \ \forall t\in[0,T].
\ea
$$
Plugging this result into the equations of $(Y_i,Z_i)$, we have
\bel{system-of-BSDEs}
\left\{\ba{ll}
\ns\ds d Y_1=-\big[\dbA_1Y_1+\dbC_1 Z_1+\dbB_1 Y_2\big]ds+Z_1dW(s),\\
\ns\ds d Y_2=-\big[\dbA_2Y_1+\dbC_2 Z_1+\dbB_2 Y_2\big]ds+Z_2dW(s),\\
\ns\ds Y_1(T)=0,\ Y_2(T)=0,
\ea\right.
\ee
where $\cG_3$, $\cG_4$ are bounded and defined in (\ref{definitions-of-c-G}), and
$$
\ba{ll}
\ns\ds \dbA_1:= A^{\top}+\cG_3\sB^{\top},\ \
\dbC_1:=C^{\top}+\cG_3\sD^{\top}, \ \ \dbB_1:=\cG_3 \sB^{\top},\\
\ns\ds \ \ \dbA_2:=\wt A^{\top}+\cG_4 \sB^{\top}, \ \ \dbB_2:=\sA+\cG_4 \sB^{\top},\ \ \dbC_2:=\cG_4 \sD^{\top}.
\ea
$$
If we denote by $\dbY:=(Y_1,Y_2)^{\top}$, $\dbZ:=(Z_1,Z_2)^{\top}$, and
$$
\ba{ll}
\ns\ds \cA:= \left[\ba{ll}
\ns\ds \dbA_1, \qq \dbB_1 \\
\ns\ds \dbA_2,\qq \dbB_2\\
\ea\right],\ \  \cC:= \left[\ba{ll}
\ns\ds \dbC_1,\qq 0\\
\ns\ds \dbC_2,\qq 0\\
\ea\right],
\ea
$$
we can rewrite (\ref{system-of-BSDEs}) as
$$
\ba{ll}
\ns\ds d\dbY(s)=\big[\cA(s) \dbY(s)+\cC(s)\dbZ(s)\big]ds+\dbZ(s)dW(s),\ \ \dbY(T)=0.
\ea
$$
Since $\cA(\cd)$, $\cC(\cd)$ are  bounded and deterministic, hence $\dbY(\cd)\equiv0$, and thus $\bar \cK_d(\cd)=\bar \cH(\cd)\equiv0.$
\end{proof}
\br
The uniqueness of open-loop equilibrium in Markovian setting was also studied in Section 4 of \cite{Hu-Jin-Zhou-2017}. In contrast, we obtain the similar uniqueness conclusions by a different approach under the general mean-field framework.
\er

\section{Concluding remarks}

In this paer, a class of time inconsistent stochastic linear quadratic problems is discussed where the state equation is described by a controlled linear conditional mean-field stochastic differential equations (SDEs). Since the mean-field terms in state equations also lead to time inconsistency, both open-loop equilibrium controls and their closed-loop representations have to be redefined in new manners. The characterizations are established for previous two notions and several new features are revealed as well. An interesting result, i.e., Lemma \ref{Lemma-equality}, and several remarks in Section 3 are given to explain the essential difference with the particular case of controlled SDEs. The relevant study on closed-loop equilibrium controls/strategies and related Riccati equations is much more complicated, and we hope to discuss it in our future publications.

%

%
%
%
%

\end{document}